\documentclass[11pt,a4paper]{article}

\usepackage{epsf,epsfig,amsfonts,amsgen,amsmath,amstext,amsbsy,amsopn,amsthm
}
\usepackage{amsmath,times,mathptmx}
\usepackage{amsfonts,amsthm,amssymb}
\usepackage{amsfonts}
\usepackage{graphics}
\usepackage{latexsym,bm}
\usepackage{amsfonts,amsthm,amssymb,bbding}
\usepackage{indentfirst}
\usepackage{graphicx}
\usepackage{color}
\usepackage[colorlinks=true,anchorcolor=blue,filecolor=blue,linkcolor=blue,urlcolor=blue,citecolor=blue]{hyperref}
\usepackage{float}
\usepackage{tikz}
\evensidemargin=\oddsidemargin\topmargin=-1.5cm

\newtheorem{thm}{Theorem}[section]

\newtheorem{lem}{Lemma}[section]
\newtheorem{cor}{Corollary}[section]

\newtheorem{conj}{Conjecture}[section]
\newtheorem{claim}{Claim}[section]
\newtheorem{definition}{Definition}[section]

\addtocounter{section}{0}

\begin{document}
\title{ Extremal spectral results of planar graphs without vertex-disjoint cycles\footnote{Lin was supported by
NSFC grant 12271162 and Natural Science Foundation of Shanghai (No. 22ZR1416300). Shi was supported by NSFC grant 12161141006.}}
\author{{\bf Longfei Fang$^{a,b}$}, {\bf Huiqiu Lin$^a$}\thanks{Corresponding author: huiqiulin@126.com
(H. Lin).}, {\bf Yongtang Shi$^c$} \\
\small $^{a}$ School of Mathematics, East China University of Science and Technology, \\
\small  Shanghai 200237, China\\
\small $^{b}$ School of Mathematics and Finance, Chuzhou University, \\
\small  Chuzhou, Anhui 239012, China\\
\small $^{c}$Center for Combinatorics and LPMC, Nankai University, Tianjin 300071, China\\
}

\date{}
\maketitle
{\flushleft\large\bf Abstract}
Given a planar graph family $\mathcal{F}$, let ${\rm ex}_{\mathcal{P}}(n,\mathcal{F})$ and ${\rm spex}_{\mathcal{P}}(n,\mathcal{F})$ be the maximum size and maximum spectral radius over all $n$-vertex $\mathcal{F}$-free planar graphs, respectively.
Let $tC_{\ell}$ be the disjoint union of $t$ copies of $\ell$-cycles,
and $t\mathcal{C}$ be the family of $t$ vertex-disjoint cycles without length restriction. Tait and Tobin [Three conjectures in extremal spectral graph theory,
J. Combin. Theory Ser. B 126 (2017) 137--161] determined that $K_2+P_{n-2}$ is the extremal spectral graph among all planar graphs with sufficiently large order $n$, which implies the extremal graphs of both ${\rm spex}_{\mathcal{P}}(n,tC_{\ell})$ and ${\rm spex}_{\mathcal{P}}(n,t\mathcal{C})$ for $t\geq 3$ are $K_2+P_{n-2}$. In this paper, we first determine ${\rm spex}_{\mathcal{P}}(n,tC_{\ell})$ and ${\rm spex}_{\mathcal{P}}(n,t\mathcal{C})$ and characterize the unique extremal graph for $1\leq t\leq 2$, $\ell\geq 3$ and sufficiently large $n$.
Secondly, we obtain the exact values of ${\rm ex}_{\mathcal{P}}(n,2C_4)$ and ${\rm ex}_{\mathcal{P}}(n,2\mathcal{C})$,
which solve a conjecture of Li [Planar Tur\'an number of the disjoint union of cycles, Discrete Appl. Math. 342 (2024) 260--274] for $n\geq 2661$.
\begin{flushleft}
\textbf{Keywords:} Spectral radius; Tur\'{a}n number; Planar graph; Vertex-disjoint cycles; Quadrilateral
\end{flushleft}
\textbf{AMS Classification:} 05C35; 05C50

\section{Introduction}

Given a graph family $\mathcal{F}$,
a graph is said to be \textit{$\mathcal{F}$-free}
if it does not contain any $F\in\mathcal{F}$ as a subgraph.
When $\mathcal{F}=\{F\}$, we write $F$-free instead of $\mathcal{F}$-free.
One of the earliest results in extremal graph theory is the Tur\'{a}n's theorem,
which gives the maximum number of edges in an $n$-vertex $K_k$-free graph.
The \emph{Tur\'{a}n number} ${\rm ex}(n,\mathcal{F})$ is the maximum number of edges in an $\mathcal{F}$-free graph on $n$ vertices.
Much attention has been given to Tur\'{a}n numbers on cycles.
F\"{u}redi and Gunderson \cite{Furedi} determined ${\rm ex}(n,C_{2k+1})$ for all $n$ and $k$.
However, the exact value of ${\rm ex}(n,C_{2k})$ is still open.
Erd\H{o}s \cite{Erdos} determined ${\rm ex}(n,tC_3)$ for $n\geq 400(t-1)^2$,
and the unique extremal graph is characterized.
Subsequently, Moon \cite{Moon} showed that Erd\H{o}s's result is still valid whenever $n>\frac{9t-11}{2}$.
Erd\H{o}s and P\'{o}sa \cite{Erd} showed that ${\rm ex}(n,t\mathcal{C})=(2t-1)(n-t)$ for $t\geq 2$ and $n\geq 24t$.
For more results on Tur\'{a}n-type problem, we refer the readers to the survey paper \cite{KC}.
In this paper, we initiate the  study of spectral extremal values and Tur\'{a}n numbers on cycles in planar graphs.


\subsection{Planar spectral extremal values on cycles} \label{sub1.1}

One extension of the classical Tur\'{a}n number is to study extremal spectral radius in a planar graph with a forbidden structure.
The planar spectral extremal value of a given graph family $\mathcal{F}$, denoted by ${\rm spex}_{\mathcal{P}}(n,\mathcal{F})$,
is the maximum spectral radius over all $n$-vertex $\mathcal{F}$-free planar graphs.
An $\mathcal{F}$-free planar graph on $n$ vertices with maximum spectral radius is called an \textit{extremal graph} to ${\rm spex}_{\mathcal{P}}(n,\mathcal{F})$.
Boots and Royle \cite{Boots} and independently Cao and Vince \cite{Cao} conjectured that $K_2+P_{n-2}$ is the unique planar graph with the maximum spectral radius
where `+' means the join product.
The conjecture was finally proved by Tait and Tobin \cite{Tait} for sufficiently large $n$.

In order to study the spectral extremal problems on planar graphs, we first give a useful theorem which will be frequently used in the following.
\begin{thm}\label{theorem1.1}
Let $F$ be a planar graph and $n\geq \max\{2.16\times 10^{17}, 2|V(F)|\}$.
If $F$ is a subgraph of $2K_1+P_{n/2}$ but not of $K_{2,n-2}$,
then the extremal graph to ${\rm spex}_{\mathcal{P}}(n,F)$ contains a copy of $K_{2,n-2}$.
\end{thm}

Let $tC_{\ell}$ be the disjoint union of $t$ copies of $\ell$-cycles,
and $t\mathcal{C}$ be the family of $t$ vertex-disjoint cycles without length restriction.
For $t\geq 3$, it is easy to check that $K_2+P_{n-2}$ is $tC_{\ell}$-free and $t\mathcal{C}$-free.
Theorem \ref{theorem1.1} implies that $K_2+P_{n-2}$ is the extremal graph to ${\rm spex}_{\mathcal{P}}(n,tC_{\ell})$ and
${\rm spex}_{\mathcal{P}}(n,t\mathcal{C})$ for $t\geq 3$ and sufficiently large $n$.
For three positive integers $n,n_1,n_2$ with $n_1\geq n_2$ and $n\geq n_1+2n_2+2$,
we can find integers $\alpha$ and $\beta$ such that $1\leq \beta \leq n_2$ and $n-2=n_1+\alpha n_2+\beta$.
Let
$$H(n_1,n_2):=P_{n_1}\cup \alpha P_{n_2}\cup P_{\beta}.$$

In the paper, we give answers to ${\rm spex}_{\mathcal{P}}(n,tC_{\ell})$ for $t\in \{1,2\}$ as follows.

\begin{thm}\label{theorem1.2}
For integers $\ell\geq 3$ and $n\geq \max\{2.16\times 10^{17},9\times2^{\ell-1}+3\}$, the graph $K_2+H(2\ell-3,\ell-2)$ is the extremal graph to ${\rm spex}_{\mathcal{P}}(n,2C_{\ell})$.
\end{thm}


Theorem \ref{theorem1.4} implies that $K_2+H(3,1)$ is $2\mathcal{C}$-free for $n\geq 5$.
By Theorem \ref{theorem1.2}, $K_2+H(3,1)$ is the extremal graph to ${\rm spex}_{\mathcal{P}}(n,2C_3)$ for $n\geq 2.16\times 10^{17}$.
Moreover, a planar graph is $2C_3$-free when it is $2\mathcal{C}$-free.
Hence, one can easily get answer to ${\rm spex}_{\mathcal{P}}(n,2\mathcal{C})$.

\begin{cor}\label{cor1.1}
For $n\geq 2.16\times 10^{17}$, $K_2+H(3,1)$ is the extremal graph to ${\rm spex}_{\mathcal{P}}(n,2\mathcal{C})$.
\end{cor}

We use $J_n$ to denote the graph obtained from $K_1+(n-1)K_1$
by embedding a maximum matching within its independent set.
Nikiforov \cite {Nikiforov5} and  Zhai and Wang \cite{Zhai-3} showed that $J_n$ is the extremal graph to ${\rm spex}(n,C_4)$ for odd and even $n$, respectively.
Clearly, $J_n$ is planar.
This implies that $J_n$ is the extremal graph to ${\rm spex}_{\mathcal{P}}(n,C_4)$.
Nikiforov \cite{Nikiforov2} and Cioab\u{a}, Desai, and Tait \cite{Cioaba1} determined the spectral extremal graph among $C_{\ell}$-free graphs for odd $\ell$ and even $\ell\geq6$, respectively.
Next we give the characterization of the spectral extremal graphs among $C_{\ell}$-free planar graphs.

\begin{thm}\label{theorem1.3}
For integers $\ell\geq 3$ and $n\geq \max\{2.16\times 10^{17},9\times2^{\lfloor\frac{\ell-1}{2}\rfloor}+3,\frac{625}{32}\lfloor\frac{\ell-3}{2}\rfloor^2+2\}$,\\
(i) $K_{2,n-2}$ is the unique extremal graph to ${\rm spex}_{\mathcal{P}}(n,C_{3})$ for $\ell=3$;\\
(ii) $K_2+H(\lceil\frac{\ell-3}{2}\rceil,\lfloor\frac{\ell-3}{2}\rfloor)$ is the unique extremal graph to ${\rm spex}_{\mathcal{P}}(n,C_{\ell})$ for $\ell\geq 5$.
\end{thm}

\subsection{Planar  Tur\'{a}n numbers on cycles} \label{sub1.2}

We also study another extension of the classical Tur\'{a}n number, i.e., the planar Tur\'{a}n number. Dowden \cite{DZ}  initiated the following problem: what is the maximum number of edges
in an $n$-vertex $\mathcal{F}$-free planar graph?
This extremal number is called planar Tur\'{a}n number of $\mathcal{F}$
and denoted by ${\rm ex}_{\mathcal{P}}(n,\mathcal{F})$. The planar Tur\'{a}n number for short cycles are studied
in \cite{CL2022,DZ,Du,DG,Gyori,LS,Shi,Shi-1}, but ${\rm ex}_{\mathcal{P}}(n,C_{\ell})$ is still open for general $\ell$.
For more results on planar Tur\'{a}n-type problem, we refer the readers to a survey
of Lan, Shi and Song \cite{B20}.
It is easy to see that ${\rm ex}_{\mathcal{P}}(n,t\mathcal{C})=n-1$ for $t=1$.
Lan, Shi and Song \cite{D} showed that ${\rm ex}_{\mathcal{P}}(n,t\mathcal{C})=3n-6$ for $t\geq 3$,
and the double wheel $2K_1+C_{n-2}$ is an extremal graph.
We prove the case of $t=2$, which will be used to prove our main theorems.

\begin{thm}\label{theorem1.4}
For $n\geq 5$, we have ${\rm ex}_{\mathcal{P}}(n,2\mathcal{C})=2n-1$.
The extremal graphs are obtained from $2K_1+C_3$ and an independent set of size $n-5$ by
joining each vertex of the independent set to any two vertices of the triangle.
\end{thm}

Moreover, Lan, Shi and Song \cite{D} also proved that ${\rm ex}_{\mathcal{P}}(n,tC_{\ell})=3n-6$ for all $t,\ell\geq 3$.
They \cite{Lan3} further showed that ${\rm ex}_{\mathcal{P}}(n,2C_3)=\lceil\frac{5n}{2}\rceil-5$
and obtained lower bounds of ${\rm ex}_{\mathcal{P}}(n,2C_{\ell})$ for $\ell\geq 4$,
which was improved by Li \cite{Li} for sufficiently large $n$ recently.
Li \cite{Li} also raised the following conjecture.

\begin{conj}
If $n\geq 23$, then ${\rm ex}_{\mathcal{P}}(n,2C_4)\leq \frac{19}{7}(n-2)$ and the bound is tight for $14\mid (n-2)$.
\end{conj}

In this paper, we determine the exact value of ${\rm ex}_{\mathcal{P}}(n,2C_4)$ for large $n$,
which solve the conjecture for $n\geq 2661$.

\begin{thm}\label{theorem1.5}
For $n\geq 2661$,
    $${\rm ex}_{\mathcal{P}}(n,2C_4)=\left\{
                          \begin{array}{ll}
                          \frac{19n}{7}-6 & \hbox{if $7\mid n$,} \\
                         \big\lfloor\frac{19n-34}{7}\big\rfloor~~~~~~~ & \hbox{otherwise.}
                          \end{array}
                          \right.
$$

\end{thm}

\section{Proof of Theorem \ref{theorem1.1}} \label{section2}
Let $A(G)$ be the adjacency matrix of a planar graph $G$,
and $\rho(G)$ be its spectral radius, i.e.,
the maximum modulus of eigenvalues of $A(G)$.
Throughout Sections \ref{section2} and \ref{section3}, let $G$ be an extremal graph to ${\rm spex}_{\mathcal{P}}(n,F)$, and $\rho$ denote this spectral radius.
By Perron-Frobenius theorem, there exists a positive eigenvector $X=(x_1,\ldots,x_n)^{\mathrm{T}}$ corresponding to $\rho$.
Choose $u'\in V(G)$ with $x_{u'}=\max\{x_i~|~i=1,2,\dots,n\}=1$.
For a vertex $u$ and a positive integer $i$,
let $N_i(u)$ denote the set of vertices at distance $i$ from $u$ in $G$.
For two disjoint subset $S,T\subset V(G)$, denote by $G[S,T]$ the bipartite subgraph of $G$
with vertex set $S\cup T$ that consist of all edges with one endpoint in $S$ and the other endpoint in $T$.
Set $e(S)=|E(G[S])|$ and $e(S,T)=|E(G[S,T])|$.
Since $G$ is a planar graph, we have
\begin{align}\label{align.-06}
e(S)\leq 3|S|-6~~~\text{and}~~~e(S,T)\leq 2(|S|+|T|)-4.
\end{align}

Now, we are ready to give the proof of Theorem \ref{theorem1.1}.

\begin{proof}
We give the proof in a sequence of claims.
Firstly, we give the lower bound of $\rho$.
\begin{claim}\label{Claim2.1}
For $n\geq 2.16\times 10^{17}$, we have $\rho\geq\sqrt{2n-4}.$
\end{claim}

\begin{proof}
Note that $K_{2,n-2}$ is planar and $F$-free.
Then, $\rho\geq \rho(K_{2,n-2})=\sqrt{2n-4}$, as $G$ is an extremal graph to ${\rm spex}_{\mathcal{P}}(n,F)$.
\end{proof}

\begin{claim}\label{claim2.2}
Set $L^{\lambda}=\{u\in V(G)~|~x_u\geq \frac{1}{10^3\lambda}\}$ for some constant $\lambda\geq \frac{1}{10^3}$.
Then $|L^{\lambda}|\le \frac{\lambda n}{10^5}$.
\end{claim}

\begin{proof}
By Claim \ref{Claim2.1}, $\rho\geq\sqrt{2n-4}$.
Hence,
$$\frac{\sqrt{2n-4}}{10^3\lambda}\leq\rho x_u=\sum_{v\in N_G(u)}x_v\le d_G(u)$$
for each $u\in L^{\lambda}$.
Summing this inequality over all vertices $u\in L^{\lambda}$, we obtain
\begin{align*}
|L^{\lambda}|\frac{\sqrt{2n-4}}{10^3\lambda}\le \sum_{u\in L^{\lambda}}d_G(u)
\le \sum_{u\in V(G)}d_G(u)
\le 2(3n-6).
\end{align*}
It follows that
$|L^{\lambda}|\le  3\times 10^3\lambda\sqrt{2n-4}\le \frac{\lambda n}{10^5}$ as $n\geq 2.16\times 10^{17}$.
\end{proof}

\begin{claim}\label{claim2.3}
We have $|L^1|\leq 6\times 10^4$.
\end{claim}

\begin{proof}
Let $u\in V(G)$ be an arbitrary vertex.
For convenience,
we use $N_i$, $L_i^{\lambda}$ and $\overline{L_i^{\lambda}}$ instead of $N_i(u)$, $N_i(u)\cap L^{\lambda}$ and $N_i(u)\setminus L^{\lambda}$, respectively.
By Claim \ref{Claim2.1}, $\rho\geq\sqrt{2n-4}$. Then
\begin{eqnarray}\label{align.7}
(2n-4)x_{u}\le \rho^2x_{u}=d_G(u)x_u+\sum_{v\in N_1}\sum_{w\in N_1(v)\setminus\{u\}}x_w.
\end{eqnarray}

Note that $N_1(v)\setminus\{u\}\subseteq N_1\cup N_2=L_1^{\lambda}\cup L_2^{\lambda}\cup\overline{L_1^{\lambda}}\cup \overline{L_2^{\lambda}}$.
We can calculate
$\sum_{v\in N_1}\sum_{w\in N_1(v)\setminus\{u\}}x_w$
according to two cases $w\in L_1^{\lambda}\cup L_2^{\lambda}$ or $w\in \overline{L_1^{\lambda}}\cup \overline{L_2^{\lambda}}$.
We first consider the case $w\in L_1^{\lambda}\cup L_2^{\lambda}$.
Clearly, $N_1=L_1^{\lambda}\cup\overline{L_1^{\lambda}}$ and $x_w\le 1$ for $w\in L_1^{\lambda}\cup L_2^{\lambda}$. We can see that
\begin{eqnarray}\label{align.8}
\sum_{v\in N_1}\sum_{w\in (L_1^{\lambda}\cup L_2^{\lambda})}\!\!x_w
\leq
\big(2e(L_1^{\lambda})+e(L_1^{\lambda},L_2^{\lambda})\big)+\sum_{v\in\overline{L_1^{\lambda}}}\sum_{w\in (L_1^{\lambda}\cup L_2^{\lambda})}\!\!\!x_w.
\end{eqnarray}
By Claim \ref{claim2.2},
we have $|L^{\lambda}|\le \frac{\lambda n}{10^5}$.
Moreover, $L_1^{\lambda}\cup L_2^{\lambda}\subseteq L^{\lambda}$.
Then, by (\ref{align.-06}), we have
\begin{eqnarray}\label{align.9}
2e(L_1^{\lambda})+e(L_1^{\lambda},L_2^{\lambda})\leq
2(3|L_1^{\lambda}|-6)+(2(|L_1^{\lambda}|+|L_2^{\lambda}|)-4)
< 8|L^{\lambda}|\leq \frac{8\lambda n}{10^5}.
\end{eqnarray}

Next, we consider the remain case $w\in \overline{L_1^{\lambda}}\cup \overline{L_2^{\lambda}}$.
Clearly, $x_w\le \frac{1}{10^3\lambda}$ for $w\in\overline{L_1^{\lambda}}\cup\overline{L_2^{\lambda}}$.
Then
\begin{eqnarray}\label{align.10}
\sum_{v\in N_1}\sum_{w\in\overline{L_1^{\lambda}}\cup\overline{L_2^{\lambda}}}\!\!\!x_w \le
\Big(e(L_1^{\lambda},\overline{L_1^{\lambda}}\cup\overline{L_2^{\lambda}})
+2e(\overline{L_1^{\lambda}})+e(\overline{L_1^{\lambda}},\overline{L_2^{\lambda}})\Big)\frac{1}{10^3\lambda}
< \frac{6n}{10^3\lambda},
\end{eqnarray}
where $e(L_1^{\lambda},\overline{L_1^{\lambda}}\cup\overline{L_2^{\lambda}})
+2e(\overline{L_1^{\lambda}})+e(\overline{L_1^{\lambda}},\overline{L_2^{\lambda}})\le 2e(G)<6n$.

Combining \eqref{align.7}-\eqref{align.10}, we obtain
\begin{eqnarray}\label{align.-14}
 (2n-4)x_{u}< d_G(u)x_u+\sum_{v\in\overline{L_1^{\lambda}}}\sum_{w\in (L_1^{\lambda}\cup L_2^{\lambda})}\!\!\!x_w+\Big(\frac{8\lambda}{10}+\frac{60}{\lambda}\Big)\frac{n}{10^4}.
\end{eqnarray}

Now we prove that $d_G(u)\geq\frac{n}{10^4}$ for each $u\in L^1$.
Suppose to the contrary that there exists a vertex $\widetilde{u}\in L^1$ with $d_G(\widetilde{u})<\frac{n}{10^4}$.
Note that $x_{\widetilde{u}}\geq \frac{1}{10^3}$ as $\widetilde{u}\in L^1$.
Setting $u=\widetilde{u}$, $\lambda=10$ and combining \eqref{align.-14}, we have
\begin{eqnarray}\label{align.-15}
\frac{2n-4}{10^3}<d_G(\widetilde{u})x_{\widetilde{u}}+\sum_{v\in\overline{L_1^{10}}}\sum_{w\in (L_1^{10}\cup L_2^{10})}x_w+\frac{14n}{10^4}.
\end{eqnarray}
By (\ref{align.-06}), we have
\begin{eqnarray*}
e\big(\overline{L_1^{10}},L_1^{10}\cup L_2^{10}\big)
<2\big(|\overline{L_1^{10}}|
+|L_1^{10}\cup L_2^{10}|\big)\leq 2\big(|N_1|
+|L^{10}|\big)
\leq\frac{4n}{10^4},
\end{eqnarray*}
where $|N_1|=d_G(\widetilde{u})<\frac{n}{10^4}$
and $|L^{10}|\leq \frac{n}{10^4}$ by Claim \ref{claim2.2}.
Combining this with $d_G(\widetilde{u})<\frac{n}{10^4}$ gives
 $$d_G(\widetilde{u})x_{\widetilde{u}}+\sum_{v\in\overline{L_1^{10}}}\sum_{w\in (L_1^{10}\cup L_2^{10})}x_w+\frac{14n}{10^4}
 \leq d_G(\widetilde{u})+e\big(\overline{L_1^{10}},L_1^{10}\cup L_2^{10}\big)+\frac{14n}{10^4}\leq \frac{19n}{10^4},$$
which contradicts \eqref{align.-15}.
Therefore, $d_G(u)\geq\frac{n}{10^4}$ for each $u\in L^1$.
Summing this inequality over all vertices $u\in L^{1}$, we obtain
$$|L^1|\frac{n}{10^4}\le \sum_{u\in L^1}d_G(u)\le 2e(G)\leq  6n,$$
which yields that $|L^1|\leq 6\times 10^4$.
\end{proof}

For convenience,
we use $L$, $L_i$ and $\overline{L_i}$ instead of $L^{1}$, $N_i(u)\cap L^{1}$ and $N_i(u)\setminus L^{1}$,
respectively.
Now we give a lower bound for degrees of vertices in $L^{1}$.
\begin{claim}\label{claim2.4}
For every $u\in L$,
we have $d_G(u)\geq(x_u -\frac{4}{1000})n$.
\end{claim}

\begin{proof}
Let $\overline{L_1}'$ be the subset of $\overline{L_1}$
in which each vertex has at least $2$ neighbors in $L_1\cup L_2$.
We first prove that $|\overline{L_1}'|\leq|L_1\cup L_2|^{2}$.
If $|L_1\cup L_2|=1$, then $\overline{L_1}'$ is empty, as desired.
It remains the case $|L_1\cup L_2|\geq 2$. Suppose to the contrary that $|\overline{L_1}'|> |L_1\cup L_2|^{2}$.
Since there are only $\binom{|L_1\cup L_2|}{2}$ options for vertices in $\overline{L_1}'$ to choose a set of $2$ neighbors from $L_1\cup L_2$,
we can find a set of $2$ vertices in $L_1\cup L_2$ with at least $\Big\lceil|\overline{L_1}'|/\binom{|L_1\cup L_2|}{2}\Big\rceil\geq3$ common neighbors in $\overline{L_1}'$.
Moreover, one can observe that $u\notin L_1\cup L_2$
and $\overline{L_1}'\subseteq\overline{L_1}\subseteq N_1(u)$.
Hence, $G$ contains a copy of $K_{3,3}$, contradicting that $G$ is planar.
Hence, $|\overline{L_1}'|\leq|L_1\cup L_2|^{2}$.
It follows that
\begin{eqnarray*}\label{align31}
e(\overline{L_1},L_1\cup L_2)\leq
|\overline{L_1}\setminus \overline{L_1}'|\time 2
+|L_1\cup L_2||\overline{L_1}'|
\leq d_G(u)+(6\times10^4)^{3}
\leq d_G(u)+\frac{n}{1000},
\end{eqnarray*}
where the second last inequality holds as $|\overline{L_1}'|\leq|L_1\cup L_2|^{2}$ and $|L_1\cup L_2|\leq |L|\leq 6\times 10^4$, and the last inequality holds as $n\geq 2.16\times 10^{17}$.
Setting $\lambda=1$ and combining the above inequality with \eqref{align.-14}, we have
\begin{eqnarray*}
 (2n-4)x_u\leq d_G(u)+\Big(d_G(u)+\frac{n}{10^3}\Big)+\frac{61n}{10^4},
\end{eqnarray*}
which yields that $d_G(u)\geq(n-2)x_u-\frac{71n}{2\times 10^4}\geq (x_u-\frac{4}{1000})n$.
\end{proof}

\begin{claim}\label{claim2.5}
There exists a vertex $u''\in L_1\cup L_2$ such that $x_{u''}\geq \frac{988}{1000}$.
\end{claim}

\begin{proof}
Setting $u=u'$, $\lambda=1$ and combining \eqref{align.-14}, we have
\begin{eqnarray*}
    2n-4<d_G(u')+\sum_{v\in\overline{L_1}}\sum_{w\in L_1\cup L_2}\!\!\!x_w+\frac{61n}{10^4},
\end{eqnarray*}
which yields that
$$\sum_{v\in\overline{L_1}}\sum_{w\in L_1\cup L_2}x_w\geq 2n-4-\frac{61n}{10^4}-d_G(u')\geq\frac{993n}{1000}.$$
From Claim \ref{claim2.4} we have $d_G(u')\geq \frac{996n}{1000}$ as $u'\in L$.
For simplicity, set $N_{L_1}(u')=N_{G}(u')\cap L_1$ and $d_{L_1}(u')=|N_{L_1}(u')|$.
By Claim \ref{claim2.3}, $|L|\leq 6\times 10^4$.
Then, $d_{L_1}(u')\leq |L_1|\leq |L|\leq \frac{n}{1000}$ as $n\geq 2.16\times 10^{17}$.
It infers that $$d_{\overline{L_1}}(u')=d_G(u')-d_{L_1}(u')\geq d_G(u')-|L|\geq \frac{995n}{1000}.$$
Combining this with \eqref{align.-06} gives
$$e(\overline{L_1}, L_1\cup L_2)\le e(\overline{L_1},L)-d_{\overline{L_1}}(u')\leq (2n-4)-\frac{995n}{1000}\leq \frac{1005n}{1000}. $$
By averaging, there is a vertex $u''\in L_1\cup L_2$ such that
      $$x_{u''}\geq \frac{\sum_{v\in\overline{L_1}}\sum_{w\in (L_1\cup L_2)}x_w}{e(\overline{L_1}, L_1\cup L_2)}\geq \frac{\frac{993n}{1000}}{\frac{1005n}{1000}}
      \geq \frac{988}{1000},$$
as desired.
\end{proof}

Notice that $x_{u'}=1$ and $x_{u''}\geq \frac{988}{1000}$.
By Claim \ref{claim2.4}, we have
\begin{eqnarray}\label{align.a108}
 d_G(u')\geq \frac{996}{1000}n~~~\text{and}~~~d_G(u'')\geq \frac{984}{1000}n.
\end{eqnarray}
Now, let $D=\{u',u''\}$, $R=N_G(u')\cap N_G(u'')$,
and $R_1=V(G)\setminus (D\cup R)$.
Thus, $|R_1|\leq (n-d_G(u'))+(n-d_G(u''))\leq \frac{2n}{100}$.
Now, we prove that the eigenvector entries of vertices in $R\cup R_1$ are small.

\begin{claim}\label{claim2.6}
Let $u\in R\cup R_1$. Then $x_u\le \frac{1}{10}$.
\end{claim}

\begin{proof}
For any vertex $u\in R_1$, we can see that
\begin{eqnarray}\label{align.a8}
 d_D(u)\leq 1~~~\text{and}~~~d_{R}(u)\leq 2.
\end{eqnarray}
In fact, if $d_{R}(u)\geq 3$, then $G[N_G(u)\cup D]$ contains a copy of $K_{3,3}$, contradicting that $G$ is planar.
By (\ref{align.a8}), $d_G(u)=d_D(u)+d_{R}(u)+d_{R_1}(u)\leq 3+d_{R_1}(u)$.
Note that $|R_1|\le \frac{2n}{100}$ and $e(R_1)\leq 3|R_1|$ by \eqref{align.-06}.
Thus,
  $$\rho\sum_{u\in R_1}x_u\le \sum_{u\in R_1}d_G(u)\le \sum_{u\in R_1}(3+d_{R_1}(u))\leq 3|R_1|+2e(R_1)\le 9|R_1|\le \frac{18n}{100},$$
which yields $\sum_{u\in R_1}x_u\le \frac{18n}{100\rho}$.
For any $u\in R\cup R_1$, $d_{R}(u)\leq 2$ as $G$ is $K_{3,3}$-free.
It follows that
        $$\rho x_u=\sum_{v\in N_G(u)}x_v\leq  \sum_{v\in N_D(u)}x_v+\sum_{v\in N_{R}(u)}x_v+\sum_{v\in  N_{R_1}(u)}x_v\le 4+\frac{18n}{100\rho}.$$
Dividing both sides by $\rho$, we get $x_u\le \frac{1}{10}$  as $\rho\geq \sqrt{2n-4}$.
\end{proof}

\begin{figure}[!ht]
	\centering
	\includegraphics[width=0.35\textwidth]{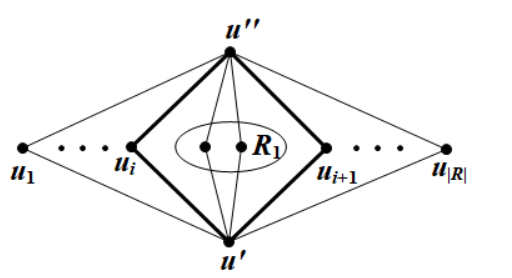}
	\caption{A local structure of $G'$. }{\label{fig.-12}}
\end{figure}

\begin{claim}\label{claim2.7}
$R_1$ is empty.
\end{claim}

%
\begin{proof}
Suppose to the contrary that $a=|R_1|>0$.
Assume that $G^*$ is a planar embedding of $G[D\cup R]$, and  $u_1,\dots,u_{|R|}$ are
around $u'$ in clockwise order in $G^*$, with subscripts interpreted modulo $|R|$.
Recall that $|R_1|\leq \frac{2n}{100}$.
Then $|R|=n-2-|R_1|> \frac{97n}{100}>|R_1|$.
By the pigeonhole principle, there exists an integer $i\le |R|$ such that $u'u_iu''u_{i+1}u'$ is a face of the plane graph $G^*$.
We modify the graph $G^*$ by  joining each vertex in $R_1$ to each vertex in $D$ and making these edges cross the face $u'u_iu''u_{i+1}u'$,
to obtain the graph $G'$ (see $G'$ in Figure \ref{fig.-12}). Then, $G'$ is a plane graph.

We first show that $G'$ is $F$-free.
Suppose to the contrary, then $G'$ contains a subgraph $F'$ isomorphic to $F$.
From the modification, we can see that $V(F')\cap R_1$ is not empty.
Moreover, since $|R|>\frac{97n}{100}> |V(F)|=|V(F')|$, we have
$$|R\setminus V(F')|=|R|-|R\cap V(F')|>|V(F')|-|V(F')\cap R|\geq |V(F')\cap R_1|.$$
Then, we may assume that $V(F')\cap R_1=\{v_1,\dots,v_b\}$ and $\{w_1,\dots,w_b\}\in R\setminus V(F')$.
Clearly, $N_{G'}(v_i)=D\subseteq N_{G'}(w_i)$ for each $i\in \{1,\dots,b\}$.
This indicates that a copy of $F$ is already present in $G$, a contradiction.
Therefore, $G'$ is $F$-free.

In what follows, we shall obtain a contradiction by showing that $\rho(G')>\rho$.
Clearly, $G[R_1]$ is  planar. Then, there exists a vertex $v_1\in R_1$ with $d_{R_1}(v_1)\leq 5$.
Define $R_2=R_1\setminus\{v_1\}$.
Repeat this step, we obtain a sequence of sets $R_1,R_2,\cdots,R_{a}$
with $d_{R_{i}}(v_{i})\leq 5$ for each $i\in\{1,\ldots,a\}$.
Combining this with (\ref{align.a8}), we have
\begin{eqnarray}\label{align.13}
    \sum_{w\in N_{D\cup R\cup R_i}(v_i)}x_w\le 1+\sum_{w\in N_R(v_i)}x_w+\sum_{w\in N_{R_i}(v_i)}x_w\le \frac{17}{10},
\end{eqnarray}
where the last inequality holds as $x_w<\frac{1}{10}$ for any $w\in R\cup R_1$ by Claim \ref{claim2.6}.
It is not hard to verify that in graph $G$ the set of edges incident to vertices in $R_1$ is $\bigcup_{i=1}^{a}\{wv_i~|~w\in N_{D\cup R\cup R_i}(v_i)\}$. Note that $x_{u'}+x_{u''}\geq \frac{1988}{1000}$.
Combining these with (\ref{align.13}), we have
$$\rho(G')-\rho\geq \frac{2}{X^{\mathrm{T}}X}\sum_{i=1}^{a}x_{v_i}\left((x_{u'}+x_{u''})-\sum_{w\in N_{D\cup R\cup R_i}(v_i)}x_w\right)>0,$$
which contradicts that $G$ is extremal to ${\rm spex}_{\mathcal{P}}(n,F)$.
Therefore, $R_1$ is empty.
\end{proof}

By Claim \ref{claim2.7}, $G$ contains a copy of $K_{2,n-2}$.
The proof of Theorem \ref{theorem1.1} is complete.
\end{proof}

\section{Proofs of Theorems \ref{theorem1.2} and \ref{theorem1.3}}\label{section3}

Recall that $G$ is an extremal graph to ${\rm spex}_{\mathcal{P}}(n,F)$.
In the case $F=C_3$, by Theorem \ref{theorem1.1}, $G$ contains a copy of $K_{2,n-2}$.
We further obtain that $G\cong K_{2,n-2}$ as $G$ is triangle-free
(otherwise, adding any edge increases triangles, a contradiction).
Now we consider the case $F\in \{C_{\ell}~|~\ell\geq 5\}\cup \{2C_{\ell}~|~\ell\geq 3\}$.
Let $P_n$ denote a path on $n$ vertices.
For convenience, an isolated vertex is referred to as a path of order 1.
We first give two lemmas to characterize the structural features of the extremal graph $G$,
which help us to present an approach to prove Theorems \ref{theorem1.2} and \ref{theorem1.3}.
\begin{lem}\label{lemma3.1}
Let $n\geq \max\{2.16\times 10^{17},2|V(F)|\}$ and $F\in \{C_{\ell}~|~\ell\geq 5\}\cup \{2C_{\ell}~|~\ell\geq 3\}$.
Then $G\cong K_2+G[R]$ and $G[R]$ is a disjoint union of paths.
\end{lem}

\begin{proof}
Since $n\geq 2|V(F)|$ and $F\in \{C_{\ell}~|~\ell\geq 5\}\cup \{2C_{\ell}~|~\ell\geq 3\}$,
we can see that $F$ is a subgraph of $2K_1+P_{n/2}$ but not of $K_{2,n-2}$.
Hence, by Theorem \ref{theorem1.1}, $G$ contains a copy of $K_{2,n-2}$,
and $u',u''$ are the vertices of degree $n-2$ and $R$ is the set of vertices of degree two in $K_{2,n-2}$.
\begin{claim}\label{Claim3.1}
$u'u''\in E(G)$.
\end{claim}

\begin{proof}
Suppose to the contrary that $u'u''\notin E(G)$.
Assume that $G^*$ is a planar embedding of $G$, and  $u_1,\dots,u_{n-2}$ are
around $u'$ in clockwise order in $G^*$, with subscripts interpreted modulo $n-2$.
If $R$ induces an cycle $u_1u_2\dots u_{n-2}u_1$,
then we can easily check that $G^*$ contains a copy of $F$, a contradiction.
Thus, there exists an integer $i\le n-2$ such that $u_iu_{i+1}\notin E(G^*[R])$.
Furthermore, $u'u_iu''u_{i+1}u'$ is a face in $G^*$.

We modify the graph $G^*$ by adding the edge $u'u''$ and making $u'u''$ cross the face $u'u_iu''u_{i+1}u'$,
to obtain the graph $G'$.
Clearly, $G'$ is a plane graph.
We shall show that $G'$ is $F$-free.
Suppose to the contrary that $G'$ contains a subgraph $F'$ isomorphic to $F$.
If $F'=C_{\ell}$ for some $\ell\geq 5$,
then $G'$ contains an $\ell$-cycle containing $u'u''$, say $u'u''u_1'u_2'\dots u_{\ell-2}'u'$.
However, an ${\ell}$-cycle $u'u_1'u''u_2'\dots u_{\ell-2}'u'$ is already present in $G$, a contradiction.
If $F'=2C_{\ell}$ for some $\ell\geq 3$,
then $F'$ contains two vertex-disjoint $\ell$-cycles $C^1$ and $C^2$.
From the modification, we can see that one of $C^1$ and $C^2$ (say $C^1$) contains the edge $u'u''$.
This implies that $C^2$ is a subgraph of $G[R]$.
However, $G[V(C^2)\cup \{u',u''\}]$ contains a $K_5$-minor, contradicting the fact that $G$ is planar.
Hence, $G'$ is $F$-free.
However, $\rho(G')>\rho$, contradicting that $G$ is extremal to ${\rm spex}_{\mathcal{P}}(n,F)$.
Therefore, $u'u''\in E(G)$.
\end{proof}

From Claim \ref{Claim3.1} we know that $G=K_2+G[R]$.
It remains to show that $G[R]$ is a disjoint union of paths.
Theorem \ref{theorem1.1} and Claim \ref{Claim3.1} imply that $u'$ and $u''$ are dominating vertices.
Furthermore, since $G$ is $K_5$-minor-free and $K_{3,3}$-minor-free,
we can see that $G[R]$ is $K_3$-minor-free and $K_{1,3}$-minor-free.
This implies that $G[R]$ is an acyclic graph with maximum degree at most 2.
Thus, $G[R]$ is a disjoint union of paths.
The result follows.
\end{proof}

Now we give a transformation that we will use in subsequent proof.
\begin{definition}\label{def3.1}
Let $s_1$ and $s_2$ be two integers with $s_1\geq s_2\geq 1$,
and let $H=P_{s_1}\cup P_{s_2}\cup H_0$, where $H_0$ is a disjoint union of paths.
We say that $H^*$ is an $(s_1,s_2)$-transformation of $H$ if
  $$H^*:= \left\{
                          \begin{array}{ll}
                          P_{s_1+1}\cup P_{s_2-1}\cup H_0 & \hbox{if $s_2\geq 2$,} \\
                          P_{s_1+s_2}\cup H_0~~~~~~~ & \hbox{if $s_2= 1$.}
                          \end{array}
                          \right.
  $$
\end{definition}
Clearly, $H^*$ is a disjoint union of paths, which implies that $K_2+H^*$ is planar.
If $G[R]\cong H$, then we shall show that $\rho(K_2+H^*)>\rho$ for sufficiently large $n$.
\begin{lem}\label{lemma3.2}
Let $H$ and $H^*$ be defined as in Definition \ref{def3.1}, $n\geq \max\{2.16\times 10^{17},2|V(F)|,9\times2^{s_2+1}+3\}$ and $F\in \{C_{\ell}~|~\ell\geq 5\}\cup \{2C_{\ell}~|~\ell\geq 3\}$.
If $G[R]\cong H$, then $\rho(K_2+H^*)>\rho$.
\end{lem}

\begin{proof}
By Lemma \ref{lemma3.1}, $G\cong K_2+G[R]$ and $G[R]$ is a disjoint union of paths.
If $G[R]\cong H$, then $G=K_2+H$.
Assume that $P^1:=v_1v_2\dots v_{s_1}$ and $P^2:=w_1w_2\dots w_{s_2}$ are two components of $H$.
If $s_2=1$, then $H\subset H^*$, and so $G\subset K_2+H^*$.
It follows that $\rho(P_2+H^*)>\rho$, the result holds.
Next, we deal with the case $s_2=2$.
If $x_{v_1}\leq x_{w_1}$, then let $H'$ be obtained from $H$ by deleting the edge $v_{1}v_{2}$ and
adding the edge  $v_{2}w_{1}$. Clearly, $H'\cong H^*$. Moreover,
\begin{eqnarray*}
\rho(K_2+H^*)-\rho
\geq\frac{X^{\mathrm{T}}\big(A(K_2+H^*)-A(G)\big)X}{X^{\mathrm{T}}X}
\geq\frac{2}{X^{\mathrm{T}}X}(x_{w_{1}}-x_{v_{1}})x_{v_{2}}\geq 0.
\end{eqnarray*}
Since $X$ is a positive eigenvector of $G$, we have $\rho x_{v_1}=2+x_{v_2}$.
If $\rho(K_2+H^*)=\rho$, then $X$ is also a positive eigenvector of $K_2+H^*$, and so $\rho(K_2+H^*) x_{v_1}=2$, contradicting that $\rho x_{v_1}=2+x_{v_2}$.
Thus, $\rho(K_2+H^*)>\rho$, the result holds.
The case $x_{v_1}>x_{w_1}$ is similar and hence omitted here.

It remains the case $s_2\geq 3$.
We shall give characterizations of eigenvector entries of vertices in $R$ in the following two claims.

\begin{claim}\label{Claim3.2}
For any vertex $u\in R$, we have $x_u\in [\frac{2}{\rho}, \frac{2}{\rho}+\frac{6}{\rho^2}]$.
\end{claim}

\begin{proof}
By Lemma \ref{lemma3.1}, $u'$ and $u''$ are dominating vertices of $G$. So, $x_{u'}=x_{u''}=1$. Hence
\begin{eqnarray}\label{align.+14}
\rho x_u=x_{u'}+x_{u''}+\sum_{v\in N_R(u)}x_v=2+\sum_{v\in N_R(u)}x_v.
\end{eqnarray}
Moreover, by Lemma \ref{lemma3.1}, $d_{R}(v)\le 2$ for any $v\in R$.
Combining this with Claim \ref{claim2.6} and \eqref{align.+14} gives $x_u\in \big[\frac{2}{\rho},\frac{3}{\rho}\big]$.
Furthermore, again by \eqref{align.+14}, $\rho x_u\in \big[2,2+\frac{6}{\rho}\big]$, which yields that $x_u\in \big[\frac{2}{\rho}, \frac{2}{\rho}+\frac{6}{\rho^2}\big]$.
\end{proof}

\begin{claim}\label{claim3.1}
Let $i$ be a positive integer. Set $A_i=\big[\frac{2}{\rho}-\frac{6\times 2^i}{\rho^2},\frac{2}{\rho}+\frac{6\times 2^i}{\rho^2}\big]$ and $B_i=\big[-\frac{6\times 2^i}{\rho^{2}},\frac{6\times 2^i}{\rho^{2}}\big]$. Then, \\
(i) for any $i\in \{1,\dots,\lfloor\frac{s_1-1}{2}\rfloor\}$, $\rho^i(x_{v_{i+1}}-x_{v_i})\in A_i$;\\
 (ii) for any $i\in \{1,\dots,\lfloor\frac{s_2-1}{2}\rfloor\}$, $\rho^i(x_{w_{i+1}}-x_{w_i})\in B_i$;\\
(iii) for  any $i\in \{1,\dots,\lfloor\frac{s_2}{2}\rfloor\}$,
 $\rho^i(x_{v_{i}}-x_{w_i})\in B_i$.
\end{claim}
\begin{proof}
(i) It suffices to prove that for any $i\in \{1,\dots,\lfloor\frac{s_1-1}{2}\rfloor\}$,
\begin{eqnarray*}
 \rho^i(x_{v_{j+1}}-x_{v_j})\in \left\{
                                       \begin{array}{ll}
                                         A_i  & \hbox{if $j=i$,} \\
                                         B_i  & \hbox{if $i+1\leq j\leq s_1-i-1$.}
                                       \end{array}
                                     \right.
\end{eqnarray*}
We shall proceed the proof by induction on $i$. Clearly,
\begin{eqnarray}\label{align.+15}
\rho x_{v_j}=\sum_{v\in N_G(v_j)}x_v= \left\{
                                       \begin{array}{ll}
                                         2+x_{v_2}  & \hbox{if $j=1$,} \\
                                         2+x_{v_{j-1}}+x_{v_{j+1}}  & \hbox{if $2\le j\le s_1-1$.}
                                       \end{array}
                                     \right.
\end{eqnarray}
By Claim \ref{Claim3.2}, we have
\begin{eqnarray}\label{align.110}
\rho(x_{v_{j+1}}-x_{v_j})= \left\{
                                       \begin{array}{ll}
                                        x_{v_{1}}+x_{v_{3}}-x_{v_2}\in A_1  & \hbox{if $j=1$,} \\
                                         (x_{v_{j}}-x_{v_{j-1}})+(x_{v_{j+2}}-x_{j+1})
                                         \in B_1  & \hbox{if $2\le j\le s_1-2$.}
                                       \end{array}
                                     \right.
\end{eqnarray}
So the result is true when $i=1$.
Assume then that $2\le i\le \lfloor\frac{s_1-1}{2}\rfloor$, which implies that $s_1\geq 2i+1$.
For $i\le j\le s_1-i-1$, $\rho(x_{v_{j+1}}-x_{v_j})=(x_{v_{j}}-x_{v_{j-1}})+(x_{v_{j+2}}-x_{v_{j+1}})$, and so
\begin{eqnarray}\label{align.111}
\rho^i(x_{v_{j+1}}-x_{v_j})=\rho^{i-1}(x_{v_{j}}-x_{v_{j-1}})+\rho^{i-1}(x_{v_{j+2}}-x_{v_{j+1}}).
\end{eqnarray}

By the induction hypothesis, $\rho^{i-1}(x_{v_{i}}-x_{v_{i-1}})\in A_{i-1}$ and
$\rho^{i-1}(x_{v_{i+2}}-x_{v_{i+1}})\in B_{i-1}$.
Setting $j=i$ and combining \eqref{align.111}, we have $\rho^i(x_{v_{i+1}}-x_{v_i})\in A_i$, as desired.
If $i+1\le j\le s_1-i-1$, then by the induction hypothesis, $\rho^{i-1}(x_{v_{j}}-x_{v_{j-1}})\in B_{i-1}$ and $\rho^{i-1}(x_{v_{j+2}}-x_{v_{j+1}})\in B_{i-1}$.
Thus, by \eqref{align.111}, we have $\rho^i(x_{v_{j+1}}-x_{v_j})\in B_i$, as desired.
Hence the result holds.

 The proof of (ii) is similar to that of (i) and hence omitted here.

(iii) It suffices to prove that for  any $i\in \{1,\dots,\lfloor\frac{s_2}{2}\rfloor\}$ and $j\in \{i,\dots,s_2-i\}$,
 $\rho^i (x_{v_{j}}-x_{w_j})\in B_i.$
We shall proceed the proof by induction on $i$. Clearly,
$$\rho x_{w_j}=\sum_{w\in N_G(w_j)}x_w= \left\{
                                       \begin{array}{ll}
                                         2+x_{w_2}  & \hbox{if $j=1$,} \\
                                         2+x_{w_{j-1}}+x_{w_{j+1}}  & \hbox{if $2\le j\le s_2-1$.}
                                       \end{array}
                                     \right.
$$
Combining this with (\ref{align.+15}) and Claim \ref{Claim3.2} gives
$$\rho(x_{v_j}-x_{w_j})= \left\{
                                       \begin{array}{ll}
                                         x_{v_{2}}-x_{w_2}\in\big[-\frac{6}{\rho^2},\frac{6}{\rho^2}\big]\subset B_1  & \hbox{if $j=1$,} \\
                                         (x_{v_{j+1}}-x_{w_{j+1}})+(x_{v_{j-1}}-x_{w_{j-1}})
                                         \in B_1 & \hbox{if $2\le j\le s_2-1$.}
                                       \end{array}
                                     \right.
$$
So the claim is true when $i=1$.
Assume then that $2\le i\le \lfloor\frac{s_2}{2}\rfloor$, which implies that $s_2\geq 2i$.
If $i\le j\le s_2-i$, then $\rho(x_{v_{j}}-x_{w_j})=(x_{v_{j-1}}-x_{w_{j-1}})+(x_{v_{j+1}}-x_{w_{j+1}})$, and so
\begin{eqnarray}\label{align.120}
\rho^i(x_{v_{j}}-x_{w_j})
=\rho^{i-1}(x_{v_{j-1}}-x_{w_{j-1}})+\rho^{i-1}(x_{v_{j+1}}-x_{w_{j+1}}).
\end{eqnarray}
By the induction hypothesis,
$\rho^{i-1}(x_{v_{j-1}}-x_{w_{j-1}})\in B_{i-1}$ and $\rho^{i-1}(x_{v_{j+1}}-x_{w_{j+1}})\in B_{i-1}$.
Combining these with \eqref{align.120}, we have $\rho^i(x_{v_{j}}-x_{w_j})=B_i$.
Hence the result holds.
\end{proof}

Since $n\geq 9\times2^{s_2+1}+3$, we have $\rho\geq \sqrt{2n-4}> 6\times 2^{s_2/{2}}$, and so
\begin{eqnarray*}
 \frac{2}{\rho^{i+1}}-\frac{6\times 2^i}{\rho^{i+2}}> \left(\frac{2}{\rho^{i+1}}-\frac{6\times 2^i}{\rho^{i+2}}\right)-\frac{6\times 2^i}{\rho^{i+2}}>0~~\text{for any}~~i\leq \frac{s_2}{2}.
\end{eqnarray*}
Combining this with Claim \ref{claim3.1}, we obtain
\begin{eqnarray}\label{align.115}
 x_{v_{i+1}}-x_{v_{i}}\geq\frac{2}{\rho^{i+1}}-\frac{6\times 2^i}{\rho^{i+2}}>0~~\text{for any}~~i\leq \min\Big\{\frac{s_2}{2},\Big\lfloor\frac{s_1-1}{2}\Big\rfloor\Big\},
\end{eqnarray}
and
\begin{eqnarray}\label{align.16}
 x_{v_{i+1}}-x_{w_{i}}=(x_{v_{i+1}}-x_{v_{i}})+(x_{v_{i}}-x_{w_{i}})\geq \left(\frac{2}{\rho^{i+1}}-\frac{6\times 2^i}{\rho^{i+2}}\right)-\frac{6\times 2^i}{\rho^{i+2}}>0
\end{eqnarray}
for any $i\leq \min\big\{\big\lfloor\frac{s_2}{2}\big\rfloor,\big\lfloor\frac{s_1-1}{2}\big\rfloor\big\}$.
Similarly,
\begin{eqnarray}\label{align.17}
x_{w_{i+1}}>x_{w_{i}}~~\text{and}~~x_{w_{i+1}}>x_{v_{i}}~~\text{for any}~~i\leq \Big\lfloor\frac{s_2-1}{2}\Big\rfloor.
\end{eqnarray}

Recall that $s_2\geq 3$.
Let $t_1,t_2$ be integers with $t_1,t_2\geq 1$ and $t_1+t_2=s_2-1$, and $H'$ be obtained from $H$ by deleting edges $v_{t_1}v_{t_1+1},w_{t_2}w_{t_2+1}$ and
adding edges $v_{t_1}w_{t_2},v_{t_1+1}w_{t_2+1}$.
Since $t_1+t_2=s_2-1$, we can see that $H'\cong H^*$.
Moreover,
\begin{eqnarray}\label{align.15}
\rho(K_2+H^*)-\rho\geq\frac{X^{\mathrm{T}}\big(A(K_2+H^*)-A(G)\big)X}{X^{\mathrm{T}}X}
\geq\frac{2}{X^{\mathrm{T}}X}(x_{v_{t_1+1}}-x_{w_{t_2}})(x_{w_{t_2+1}}-x_{v_{t_1}}).
\end{eqnarray}

Now, we consider the following two cases to complete the proof.

\vspace{2mm}
\noindent{{\bf{Case 1.}}} $s_2$ is odd.

Set $t_1=\frac{s_2-1}{2}$. Then, $t_2=\frac{s_2-1}{2}$ as $t_1+t_2=s_2-1$.
By (\ref{align.16}) and (\ref{align.17}), we have
$x_{v_{t_1+1}}>x_{w_{t_2}}$ and $x_{w_{t_2+1}}>x_{v_{t_1}}$,
and so $\rho(K_2+H^*)>\rho$ by \eqref{align.15}, as desired.

\vspace{2mm}
\noindent{{\bf{Case 2.}}} $s_2$ is even.

We only consider the subcase $x_{w_{s_2/2}}\geq x_{v_{s_2/2}}$.
The proof of the subcase $x_{w_{s_2/2}}<x_{v_{s_2/2}}$ is similar and hence omitted here.
Set $t_1=\frac{s_2}{2}$. Then, $t_2=\frac{s_2-2}{2}$ as $t_1+t_2=s_2-1$.
Moreover, $x_{w_{t_2+1}}\geq x_{v_{t_1}}$, as $x_{w_{s_2/2}}\geq x_{v_{s_2/2}}$.
If $s_1=s_2$, then $s_1$ is even.
This implies that $x_{v_{(s_1+2)/2}}=x_{v_{s_1/2}}$ by symmetry, that is, $x_{v_{t_1+1}}=x_{v_{t_1}}$.
If $s_1\geq s_2+1$, then by \eqref{align.115}, $x_{v_{t_1+1}}>x_{v_{t_1}}$.
In both situations, we have $x_{v_{t_1+1}}\geq x_{v_{t_1}}$.
From (\ref{align.16}) we have $x_{v_{t_1}}>x_{w_{t_2}}$, which implies that $x_{v_{t_1+1}}>x_{w_{t_2}}$.
Furthermore, we have $\rho(K_2+H^*)\geq\rho$ by \eqref{align.15}.
If $\rho(K_2+H^*)=\rho$,
then $X$ is a positive eigenvector of $K_2+H^*$, and so $\rho(K_2+H^*) x_{v_{t_1}}=2+x_{v_{t_1-1}}+x_{w_{t_2}}$.
On the other hand, $\rho x_{v_{t_1}}=2+x_{v_{t_1-1}}+x_{v_{t_1+1}}$,
since $X$ is a positive eigenvector of $G$.
It follows that $x_{v_{t_1+1}}=x_{w_{t_2}}$, which is a contradiction.
Thus, $\rho(K_2+H^*)>\rho$, as desired.

This completes the proof of Lemma \ref{lemma3.2}.
\end{proof}

From Lemma \ref{lemma3.1}, we may assume that $G[R]= \cup_{i=1}^{t} P_{n_i}$, where $t\geq 2$ and $n_1\geq n_2\geq \cdots \geq n_t$.
Let $H$ be a disjoint union of $t$ paths.
We use $n_i(H)$ to denote the order of the $i$-th longest path of $H$ for any $i\in \{1,\dots,t\}$.
Having Lemmas \ref{lemma3.1} and \ref{lemma3.2},
we are ready to complete the proofs of Theorems \ref{theorem1.2} and \ref{theorem1.3}.\\

\noindent{\bf Proof of Theorem \ref{theorem1.2}.}
We first give the following claim.
\begin{claim}\label{cla5.2}
If $H$ is a disjoint union of $t\geq 2$ paths,
then $K_2+H$ is $2C_{\ell}$-free if and only if $n_1(H)\leq 2\ell-3$ and $n_2(H)\leq \ell-2$.
\end{claim}
\begin{proof}
We first claim that $K_2+H$ is $2C_{\ell}$-free if and only if $P_{n_1(H)}\cup P_{n_2(H)}$ is $2P_{\ell-1}$-free.
Equivalently, $K_2+H$ contains a copy of $2C_{\ell}$ if and only if $P_{n_1(H)}\cup P_{n_2(H)}$ contains a copy of $2P_{\ell-1}$.
Assume that $K_2+H$ contains two vertex-disjoint $\ell$-cycles $C^1$ and $C^2$, and $V(K_2)=\{u',u''\}$.
Since $H$ is acyclic, we can see that $C^i$ must contain at least one vertex of $u'$ and $u''$ for any $i\in \{1,2\}$.
Without loss of generality, assume that $u'\in V(C^1)$ and $u''\in V(C^2)$.
Then, $C^1-\{u'\}\cong C^2-\{u''\}\cong P_{\ell-1}$, and so  $H$ contains a $2P_{\ell-1}$.
We can further find that $P_{n_1(H)}\cup P_{n_2(H)}$ contains a $2P_{\ell-1}$.
Conversely, assume that $P_{n_1(H)}\cup P_{n_2(H)}$ contains two vertex-disjoint paths $P^1$ and $P^2$ such that $P^1\cong P^2\cong P_{\ell-1}$.
Thus, the subgraph induced by $V(P^1)\cup \{u'\}$ contains a copy of $C_{\ell}$.
Similarly, the subgraph induced by $V(P^2)\cup \{u''\}$ contains a copy of $C_{\ell}$.
This indicates that $K_2+H$ contains a copy of $2C_{\ell}$.
So, the claim holds.

Next, we claim that $P_{n_1(H)}\cup P_{n_2(H)}$ is $2P_{\ell-1}$-free
if and only if $n_1(H)\leq 2\ell-3$ and $n_2(H)\leq \ell-2$.
If $P_{n_1(H)}\cup P_{n_2(H)}$ is $2P_{\ell-1}$-free, then $n_1(H)\leq 2\ell-3$ (otherwise, $P_{n_1(H)}$ contains a copy of $2P_{\ell-1}$, a contradiction); $n_2(H)\leq \ell-2$ (otherwise, $P_{n_1(H)}\cup P_{n_2(H)}$ contains a copy of $2P_{\ell-1}$ as $n_1(H)\geq n_2(H)\geq \ell-1$, a contradiction).
Conversely, if $n_1(H)\leq 2\ell-3$ and $n_2(H)\leq \ell-2$, then it is not hard to verify that $P_{n_1(H)}\cup P_{n_2(H)}$ is $2P_{\ell-1}$-free.

This completes the proof of Claim \ref{cla5.2}.
\end{proof}

Recall that $n_i$ (resp. $n_i(H)$) is the order of the $i$-th longest path of $G[R]$ (resp. $H$) for any $i\in \{1,\dots,t\}$.
By Claim \ref{cla5.2}, $n_1\leq 2\ell-3$  and $n_2\leq \ell-2$.
By a direct computation, we have $9\times2^{\ell-1}+3\geq \max\{2|V(F)|,9\times2^{n_2+1}+3\}$, and hence
\begin{eqnarray}\label{align.200}
n\geq \max\{2.16\times 10^{17},2|V(F)|,9\times2^{n_2+1}+3\}.
\end{eqnarray}

We first claim that $n_1=2\ell-3$.
Suppose to the contrary that $n_1\leq 2\ell-4$.
Let $H'$ be an $(n_1,n_t)$-transformation of $G[R]$.
Clearly, $n_1(H')=n_1(H)+1\leq 2\ell-3$ and $n_2(H')=n_2\leq \ell-2$.
By Claim \ref{cla5.2}, $K_2+H'$ is $2C_{\ell}$-free.
However, by \eqref{align.200} and Lemma \ref{lemma3.2}, we have $\rho(K_2+H')>\rho$,
contradicting that $G$ is extremal to ${\rm spex}_{\mathcal{P}}(n,2C_{\ell})$.
Thus, $n_1=2\ell-3$, the claim holds.

Note that $n_i\leq n_2\leq \ell-2$ for each $i\in \{2,\dots,t-1\}$.
We then claim that $n_i=\ell-2$ for $i\in \{2,\dots,t-1\}$.
If $\ell=3$, then $n_i=1$, as $n_i\leq \ell-2$, for each $i\in \{2,\dots,t\}$ and the claim holds trivially.
Now let $\ell\geq 4$.
Suppose to the contrary, then set
      $i_0=\min\{i~|~2\leq i\leq t-1,n_i\leq\ell-3\}.$
Let $H'$ be an $(n_{i_0},n_t)$-transformation of $G[R]$.
Clearly, $n_1(H')=n_1=2\ell-3$ and $n_2(H')=\max\{n_2,n_{i_0}+1\}\leq \ell-2$.
By Claim \ref{cla5.2}, $K_2+H'$ is $2C_{\ell}$-free.
However, by \eqref{align.200} and Lemma \ref{lemma3.2}, we have $\rho(K_2+H')>\rho$,
contradicting that $G$ is extremal to ${\rm spex}_{\mathcal{P}}(n,2C_{\ell})$.
So, the claim holds.

Since $n_1=2\ell-3$, $n_i=\ell-2$ for $i\in \{2,\dots,t-1\}$ and $n_t\leq \ell-2$, we can see that $G[R]\cong H(2\ell-3,\ell-2)$.
This completes the proof of Theorem \ref{theorem1.2}.
~~~~~~~~~~~~~~~~~~~~~~~~~~~~~~~~~~~~~~~~~~~~~~~~~~~~~~~~~~$\Box$\\


\noindent{\bf Proof of Theorem \ref{theorem1.3}.}
It remains the case $\ell\geq 5$.
We first give the following claim.
\begin{claim}\label{cla5.3}
If $H$ is a disjoint union of $t\geq 2$ paths,
then $K_2+H$ is $C_{\ell}$-free if and only if $n_1(H)+n_2(H)\leq \ell\!-\!3$.
\end{claim}
\begin{proof}
One can observe that the longest cycle in $K_2+H$ is of order $n_1(H)+n_2(H)+2$.
Furthermore, $K_2+(P_{n_1(H)}\cup P_{n_2(H)})$ contains a cycle $C_{i}$ for each $i\in \{3,\dots,n_1(H)+n_2(H)+2\}$.
Consequently, $n_1(H)+n_2(H)+2\leq \ell-1$ if and only if $K_2+H$ is $C_{\ell}$-free.
Hence, the claim holds.
\end{proof}

By Claim \ref{cla5.3}, $n_1+n_2 \leq\ell-3$,
 which implies that $n_2\leq \lfloor\frac{\ell-3}{2}\rfloor$ as $n_1\geq n_2$.
By a direct computation, we have $9\times2^{\lfloor\frac{\ell-1}{2}\rfloor}+3\geq \max\{2|V(F)|,9\times2^{n_2+1}+3\}$, and hence
\begin{eqnarray}\label{align.100}
n\geq \max\{2.16\times 10^{17},2|V(F)|,9\times2^{n_2+1}+3\}.
\end{eqnarray}
Since $\ell\geq 5$, we have $\lfloor\frac{\ell-3}{2}\rfloor\geq 1$,
which implies that $\lfloor\frac{\ell-3}{2}\rfloor^2\geq \lfloor\frac{\ell-3}{2}\rfloor$ and $3\lfloor\frac{\ell-3}{2}\rfloor^2\geq 2\lfloor\frac{\ell-3}{2}\rfloor+1\geq \ell-3$.
Consequently, since $n\geq \frac{625}{32}\lfloor\frac{\ell-3}{2}\rfloor^2+2>6\lfloor\frac{\ell-3}{2}\rfloor^2+2$,
we have
\begin{eqnarray}\label{align.30}
n> \Big\lfloor\frac{\ell-3}{2}\Big\rfloor^2+2\Big\lfloor\frac{\ell-3}{2}\Big\rfloor+(\ell-3)+2\geq n_2^2+3n_2+n_1+2
\end{eqnarray}
as  $n_2\leq \lfloor\frac{\ell-3}{2}\rfloor$ and $n_1+n_2 \leq\ell-3$.
On the other hand, $n-2=\sum_{i=1}^{t}n_i\leq n_1+(t-1)n_2$,
which implies that $t\geq \frac{n-2-n_1}{n_2}+1\geq n_2+4$ by \eqref{align.30}.

We first claim that $n_1+n_2=\ell-3$.
Suppose to the contrary that $n_1+n_2 \leq\ell-4$.
Let $H'$ be an $(n_1,n_t)$-transformation of $G[R]$.
Clearly, $n_1(H')=n_1+1\leq \ell-3-n_2$ and $n_2(H')=n_2$.
Then, $n_1(H')+n_2(H')\leq \ell-3$.
By Claim \ref{cla5.3}, $K_2+H'$ is $C_{\ell}$-free.
However, by \eqref{align.100} and Lemma \ref{lemma3.2}, we have $\rho(K_2+H')>\rho$,
contradicting that $G$ is extremal to ${\rm spex}_{\mathcal{P}}(n,C_{\ell})$.
Hence, the claim holds.

Secondly, we claim that $n_i=n_2$ for $i\in \{3,\dots,t-1\}$.
Suppose to the contrary, then set $i_0=\min\{i~|~3\leq i\leq t-1,n_i\leq n_2-1\}$.
Let $H'$ be an $(n_{i_0},n_t)$-transformation of $G[R]$.
Clearly, $n_1(H')=n_1$ and $n_2(H')=\max\{n_2,n_{i_0}+1\}=n_2$,
and so $n_1(H')+n_2(H')=\ell-3$.
By Claim \ref{cla5.3}, $K_2+H'$ is $C_{\ell}$-free.
However, by \eqref{align.100} and Lemma \ref{lemma3.2}, we have $\rho(K_2+H')>\rho$, contradicting that $G$ is extremal to ${\rm spex}_{\mathcal{P}}(n,C_{\ell})$.
Hence, the claim holds.
It follows that $G=K_2+(P_{n_1}\cup(t-2)P_{n_2}\cup P_{n_t})$.

Finally, we claim that $n_2=\lfloor\frac{\ell-3}{2}\rfloor$.
Note that $1\leq n_2\leq \lfloor\frac{\ell-3}{2}\rfloor$.
Then, $n_2=\lfloor\frac{\ell-3}{2}\rfloor$ for $\ell\in \{5,6\}$.
It remains the case $\ell\geq 7$.
Suppose to the contrary, then $n_2\leq \lfloor\frac{\ell-5}{2}\rfloor$ as $n_2\leq \lfloor\frac{\ell-3}{2}\rfloor$.
We use $P^i$ to denote the $i$-th longest path of $G[R]$ for $i\in \{1,\dots,t\}$.
Recall that $t\geq n_2+4$.
Then $P^2,P^3,\dots,P^{n_2+3}$ are paths of order $n_2$.
We may assume that $P^{n_2+3}=w_1w_2\dots w_{n_2}$.
Since $n_1=\ell-3-n_2\geq \lceil\frac{\ell-3}{2}\rceil\geq 2$,
 there exists an endpoint $w'$ of $P^1$ with $w'w''\in E(P^1)$.
Let $G'$ be obtained from $G$ by
(i) deleting $w'w''$ and joining $w'$ to an endpoint of $P^{n_2+2}$;
(ii) deleting all edges of $P^{n_2+3}$ and  joining $w_i$ to an  endpoint of $P^{i+1}$ for each $i\in \{1,\dots,n_2\}$.
Then, $G'$ is obtained from $G$ by deleting $n_2$ edges and adding $n_2+1$ edges.
By Claim \ref{Claim3.2}, we obtain
$$\frac{4}{\rho^2} < x_{u_i}x_{u_{j}}<\frac{4}{\rho^2}+\frac{24}{\rho^3}+\frac{36}{\rho^4}<\frac{4}{\rho^2}+\frac{25}{\rho^3}$$
for any vertices $u_i,u_{j}\in R$.
Then $$\rho(G')-\rho\geq \frac{X^{\mathrm{T}}(A(G')-A(G))X}{X^{\mathrm{T}}X}>\frac{2}{X^{\mathrm{T}}X}\left(\frac{4(n_2+1)}{\rho^2}-\frac{4n_2}{\rho^2}-\frac{25n_2}{\rho^3}\right)>0,$$
where $n_2<\lfloor\frac{\ell-3}{2}\rfloor\leq \frac{4}{25}\sqrt{2n-4}\leq \frac{4}{25}\rho$ as $n\geq \frac{625}{32}\lfloor\frac{\ell-3}{2}\rfloor^2+2$. So, $\rho(G')>\rho$.
On the other hand,  $G'\cong K_2+(P_{n_1-1}\cup(n_2+1)P_{n_2+1}\cup(t-n_2-4)P_{n_2}\cup P_{n_t})$.
By Claim \ref{cla5.3}, $G'$ is $C_{\ell}$-free,
contradicting that $G$ is extremal to ${\rm spex}_{\mathcal{P}}(n,C_{\ell})$.
So, the claim holds.

Since $n_1+n_2=\ell-3$ and $n_2=\lfloor\frac{\ell-3}{2}\rfloor$, we have $n_1=\lceil\frac{\ell-3}{2}\rceil$.
Moreover, since $n_i=\lfloor\frac{\ell-3}{2}\rfloor$ for $i\in \{2,\dots,t-1\}$ and $n_t\leq n_2$,
we can further obtain that $G[R]\cong H(\lceil\frac{\ell-3}{2}\rceil,\lfloor\frac{\ell-3}{2}\rfloor)$,
which implies that $G\cong K_2+H(\lceil\frac{\ell-3}{2}\rceil,\lfloor\frac{\ell-3}{2}\rfloor)$.
This completes the proof of Theorem \ref{theorem1.3}.
~~~~~~~~~~~~~~~~~~~~~~~~~~~~~~~~~~$\Box$\\

\section{Proof of Theorem \ref{theorem1.4}}
Above all, we shall introduce the \emph{Jordan Curve Theorem}:
any simple closed curve $C$ in the plane partitions the rest of the plane into two
disjoint arcwise-connected open sets (see \cite{Bondy}, P. 244).
The corresponding two open sets are called the interior and the exterior of $C$.
We denote them by $int(C)$ and $ext(C)$,
and their closures by $Int(C)$  and $Ext(C)$, respectively.
A \emph{plane graph} is a planar embedding of a planar graph.
The Jordan Curve Theorem gives the following lemma.

\begin{lem}\label{lemma4.1}
Let $C$ be a cycle of a plane graph $G$, and let $x,y$ be two vertices of $G$ with $x\in int(C)$ and $y\in ext(C)$, then $xy\notin E(G)$.
\end{lem}

Let $G$ be a plane graph.
A face in $G$ of size $i$ is called an $i$-face.
Let $f_i(G)$ denote the number of $i$-faces in $G$, and let $f(G)$ denote $\sum_{i}f_i(G)$.

\begin{lem}\label{lemma4.2}\emph{(Proposition 2.5 of \cite{Bondy}, P. 250)}
Let $G$ be a planar graph, and let $f$ be an arbitrary face in some planar
embedding of $G$. Then $G$ admits a planar embedding whose outer face has the same
boundary as $f$.
\end{lem}

Let $\delta(G)$ be the minimum degree of a graph $G$.
It is well known that every graph $G$ with $\delta(G)\geq 2$ contains a cycle.
In the following, we give a more delicate characterization on planar graphs,
which contains an important structural information of the extremal graphs
in Theorem \ref{theorem1.4}.

\begin{lem}\label{lemma4.3}
Let $G$ be a plane graph on $n$ vertices with $\delta(G)\geq 3$.
Then $G$ contains two vertex-disjoint cycles unless $G\in \{2K_1+C_3,K_1+C_{n-1}\}$.
\end{lem}

\begin{proof}
We first deal with some trivial cases.
Since $\delta(G)\geq 3$, we have $n\geq 1+\delta(G)\geq 4$.
If $n=4$, then $G\cong K_1+C_3$.
If $n=5$, then $2e(G)=\sum_{v\in V(G)}d_G(v)\geq3\times5=15$, and so $e(G)\geq 8$.
On the other hand, $e(G)\le 3n-6=9$, since $G$ is planar.
Thus, $e(G)\in\{8,9\}$.
It is not hard to verify that $G\cong 2K_1+C_3$ when $e(G)=9$ and $G\cong K_1+C_4$ when $e(G)$=8, as desired.
If $G$ is not connected, then $G$ contains at least two components $G_1$ and $G_2$
with $\delta(G_i)\geq 3$ for $i\in\{1,2\}$,
which implies that each $G_i$ contains a cycle.
Thus, $G$ contains two vertex-disjoint cycles, as desired.
If $G$ has a cut vertex $v$, then $G-\{v\}$ has at least two components $G_3$ and $G_4$.
Since $\delta(G)\geq 3$, we have $\delta(G_i)\geq 2$ for $i\in\{3,4\}$,
which implies that both $G_3$ and $G_4$ contain a cycle.
Thus, $G$ also contains two vertex-disjoint cycles.

Next, we only need to consider the case that $G$ is a 2-connected graph of order $n\geq 6$.
Since $G$ is 2-connected, each face of $G$ is a cycle.
Let $C$ be a face of $G$ with minimum size $g$.
By Lemma \ref{lemma4.2},
we may assume without loss of generality that $C$ is the outer face of $G$.
Let $G_1=G-V(C)$. If $G_1$ contains a cycle,
then $G$ contains two vertex-disjoint cycles, as desired.
Now assume that $G_1$ is acyclic.
Since $\delta(G)\geq 3$, we have $2e(G)=\sum_{v\in V(G)}d_G(v)\geq 3n$.
This, together with Euler's formula $n-2=e(G)-f(G)$, gives $e(G)\le 3f(G)-6$.
On the other hand,
\begin{align*}
2e(G)=\sum_{i\geq g}if_i(G)\geq g\sum_{i\geq g}f_i(G)=gf(G).
\end{align*}
Hence, $gf(G)\le 2e(G)\le 6f(G)-12$, yielding $g\le \frac{6f(G)-12}{f(G)}<6$.
Subsequently, we shall give several claims.

\begin{claim}\label{claim4.1}
We have $g=3$.
\end{claim}

\begin{figure}[!ht]
	\centering
	\includegraphics[width=0.6\textwidth]{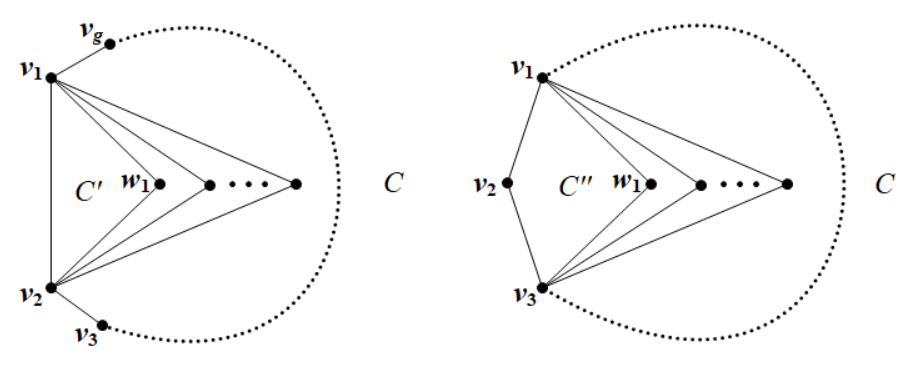}
	\caption{Two possible local structures of $G$. }{\label{fig.-01}}
\end{figure}

\begin{proof}
Suppose to the contrary that $g\in \{4,5\}$, and let $C=v_1v_2\ldots v_gv_1$.
We first consider the case that there exists a vertex of $G_1$ adjacent to two consecutive vertices of $C$.
Without loss of generality, let $w_1\in V(G_1)$ and $\{w_1,v_1,v_2\}$ induces a triangle $C'$.
More generally, we define $A=\{w\in V(G_1)~|~ v_1,v_2 \in N_C(w)\}$.
Clearly, $w_1\in A$. We can select a vertex, say $w_1$, in $A$
such that $A\subseteq Ext(C')$ (see Figure \ref{fig.-01}).
Notice that $C'$ is not a face of $G$, as $g\in \{4,5\}$.
Then, $int(C')\neq \varnothing$. By Lemma \ref{lemma4.1},
every vertex in $int(C')$ has no neighbors in $ext(C')$.
Moreover, by the definitions of $A$ and $w_1$,
every vertex in $int(C')$ has at most one neighbor in $\{v_1,v_2\}$.
It follows that every vertex in $int(C')$ has at least one neighbor in $int(C')$,
as $\delta(G)\geq3$.
Thus, $G[int(C')]$ is nonempty, that is, $G_1[int(C')]$ is nonempty.
Recall that $G_1$ is acyclic. Then $G_1[int(C')]$ contains at least two pendant vertices,
one of which (say $w_2$) is not adjacent to $w_1$.
Hence, $w_2$ is also a pendant vertex of $G_1$, as $w_2$ has no neighbors in $ext(C')$.
On the other hand, $w_2$ has at most one neighbor in $\{v_1,v_2\}$,
and so $d_C(w_2)\leq1$.
Therefore,
$d_{G}(w_2)=d_{G_1}(w_2)+d_C(w_2)\leq2$, contradicting $\delta(G)\geq3$.

Now it remains the case that
each vertex of $G_1$ is not adjacent to two consecutive vertices of $C$.
Note that $\delta(G)\geq 3$ and $G_1$ is acyclic.
Then $G_1$ contains a vertex $w_0$ with $d_{G_1}(w_0)\le 1$, and thus $d_C(w_0)= d_G(w_0)-d_{G_1}(w_0)\geq 2$.
Now, since $g\in \{4,5\}$,
we may assume without loss of generality that $v_1,v_3\in N_C(w_0)$.
Let $A'=\{w\in V(G_1)~|~ v_1,v_3 \in N_C(w)\}$.
Clearly, $w_0\in A'$ and $v_2\notin N_C(w)$ for each $w\in A'$.
Now, we can select a vertex, say $w_1$, in $A'$
such that $A'\subseteq Ext(C'')$,
where $C''=w_1v_1v_2v_3w_1$ (see Figure \ref{fig.-01}).
We can see that $int(C'')\neq\varnothing$ (otherwise, $d_G(v_2)=|\{v_1,v_3\}|=2$,
a contradiction).
By the definition of $w_1$, we have $int(C'')\cap A'=\varnothing$.
Furthermore, every vertex in $int(C'')$ has no neighbors in $ext(C'')$
and has at most one neighbor in $\{v_1,v_2,v_3\}.$
Thus, every vertex in $int(C'')$ has at least one neighbor in $int(C'')$.
By a similar argument as above, we can find a vertex $w_2\in int(C'')$ with $d_{G}(w_2)=d_{G_1}(w_2)+d_C(w_2)\leq2$,
which contradicts $\delta(G)\geq3$.
\end{proof}

By Claim \ref{claim4.1}, the outer face of $G$ is a triangle $C=v_1v_2v_3v_1$.
In the following, we denote $B_i=\{w\in V(G_1)~|~d_C(w)=i\}$ for $i\leq3$.
Since $\delta(G)\geq3$, we have $w\in B_3$ for each isolated vertex $w$ of $G_1$,
and $w\in B_2\cup B_3$ for each pendant vertex $w$ of $G_1$.

\begin{claim}\label{claim4.2}
$|B_3|\le 1$ and $|B_2|+|B_3|\geq 2$.
\end{claim}

\begin{proof}
Since $C$ is the outer face of $G$, every vertex of $G_1$ lies in $int(C)$.
Furthermore, since $G$ is planar, it is easy to see that $|B_3|\le 1$.
This implies that $G_1$ contains at most one isolated vertex.
Recall that $|G_1|=n-3\geq3$ and $G_1$ is acyclic. Then
$G_1$ contains at least two pendant vertices $w_1$ and $w_2$.
Therefore, $|B_2|+|B_3|\geq |\{w_1,w_2\}|=2$.
\end{proof}

\begin{claim}\label{claim4.3}
Let $w_0,w_1$ be two vertices in $V(G_1)$ such that
$N_C(w_0)\supseteq \{v_3\}$ and $N_C(w_1)\supseteq\{v_1,v_2\}$ (see Figure \ref{fig.-02}).
Then \\
(i) $v_3,w_0\in ext(C''')$, where $C'''=w_1v_1v_2w_1$;\\
(ii) if $w_0\notin B_3$,
then $G_1$ contains a pendant vertex in $ext(C''')$.
\end{claim}

\begin{figure}[!ht]
	\centering
	\includegraphics[width=0.2\textwidth]{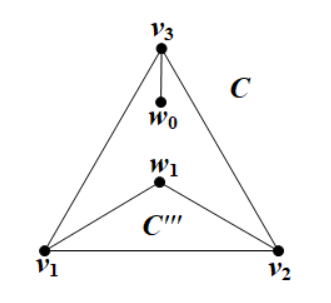}
	\caption{A local structure in Claim \ref{claim4.3}. }{\label{fig.-02}}
\end{figure}

\begin{proof}
(i) Since $C$ is the outer face and $v_3\in V(C)\setminus V(C''')$,
we have $v_3\in ext(C''')$.
Furthermore, using $w_0v_3\in E(G)$ and Lemma \ref{lemma4.1} gives $w_0\in ext(C''')$.

(ii) Since $w_0\notin B_3$, we have $d_C(w_0)\le 2$, and so
$d_{G_1}(w_0)=d_G(w_0)-d_C(w_0)\geq 1$.
By (i), we know that $w_0\in ext(C''')$.
If $d_{G_1}(w_0)=1$, then $w_0$ is a desired pendant vertex.
It remains the case that $d_{G_1}(w_0)\geq2$.
Now, whether $w_1$ is a neighbor of $w_0$ or not,
$w_0$ has at least one neighbor in $V(G_1)\cap ext(C''')$.
Thus, $G_1[ext(C''')]$ is nonempty.
Recall that $G_1$ is acyclic. Then $G_1[ext(C''')]$ contains at least two pendant vertices,
one of which (say $w_2$) is not adjacent to $w_1$.
Hence, $w_2$ is also a pendant vertex of $G_1$, as $w_2$ has no neighbors in $int(C''')$.
\end{proof}

\begin{figure}[!ht]
	\centering
	\includegraphics[width=0.45\textwidth]{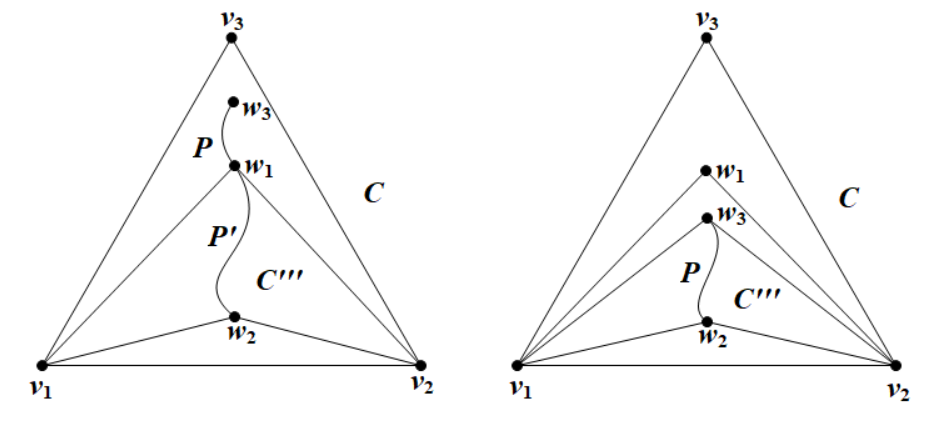}
	\caption{Two possible local structures in Claim \ref{claim4.4}. }{\label{fig.-03}}
\end{figure}

\begin{claim}\label{claim4.4}
Let $w_1,w_2\in V(G_1)$ with $N_{C}(w_1)\cap N_{C}(w_2)\supseteq \{v_1,v_2\}$.
Assume that $C'''=w_1v_1v_2w_1$ and $w_2\in int(C''')$ (see Figure \ref{fig.-03}).
Then $G$ contains a cycle $C_{v_i}$ such that $V(C_{v_i})\subseteq Int(C''')$ and $V(C_{v_i})\cap V(C)=\{v_i\}$ for each $i\in\{1,2\}$.
\end{claim}

\begin{proof}
We first claim that $N_{C}(w_2)=\{v_1,v_2\}$.
By Claim \ref{claim4.3}, we know that $v_3\in ext(C''')$.
Now, since $w_2\in int(C''')$, we have $w_2v_3\notin E(G)$ by Lemma \ref{lemma4.1},
and so $N_{C}(w_2)=\{v_1,v_2\}$.
Furthermore, we have $d_{G_1}(w_2)\geq 1$.
Then $G_1$ contains a path $P$ with endpoints $w_2$ and $w_3$, where $w_3$ is a pendant vertex of $G_1$.
If $V(P)\not\subseteq int(C''')$,
then by $w_2\in int(C''')$ and Lemma \ref{lemma4.1}, we have $V(P)\cap V(C''')=\{w_1\}$ as $v_1,v_2\notin V(G_1)$.
Now let $P'$ be the subpath of $P$ with endpoints $w_2$ and $w_1$.
Then $V(P')\setminus\{w_1\}\subseteq int(C''')$,
and $G$ contains a cycle $C(v_i)=v_iw_1P'w_2v_i$ for each $i\in\{1,2\}$, as desired.
Next, assume that $V(P)\subseteq int(C''')$.
Then, $w_3\in int(C''')$.
By $v_3\in ext(C''')$ and Lemma \ref{lemma4.1}, we get that $w_3v_3\notin E(G)$, and so $w_3\notin B_3$.
Moreover, $d_{G_1}(w_3)=1$ and $\delta(G)\geq3$ give $w_3\in B_2$.
Thus, $N_C(w_3)=\{v_1,v_2\}$.
Therefore, $G$ contains a cycle $C_{v_i}=v_iw_2Pw_3v_i$ for each $i\in\{1,2\}$, as desired.
\end{proof}

Having above four claims, we are ready to give the final proof of Lemma \ref{lemma4.3}.
By Claim \ref{claim4.2}, we have $|B_3|\leq1$ and $|B_2|\geq 1$.
We may without loss of generality that $w_1\in B_2$ and $N_{C}(w_1)=\{v_1,v_2\}$.
For each $i\in \{1,2\}$, let $\overline{i}\in \{1,2\}\setminus \{i\}$.
Since $d_{C}(w_1)=2$, we have $d_{G_1}(w_1)\geq 1$.
Hence, $G_1$ is nonempty, and so $G_1$ contains at least two pendant vertices.
According to the size of $B_3$, we now distinguish two cases to complete the proof.

\vspace{2mm}
\noindent{\bf{Case 1.}} $|B_3|= 1$.

Assume that $B_3=\{w_0\}$. Then $N_C(w_0)=\{v_1,v_2,v_3\}$ (see Figure \ref{fig.-04}).
We then consider two subcases according to the size of $B_2$.

\begin{figure}[!ht]
	\centering
	\includegraphics[width=0.7\textwidth]{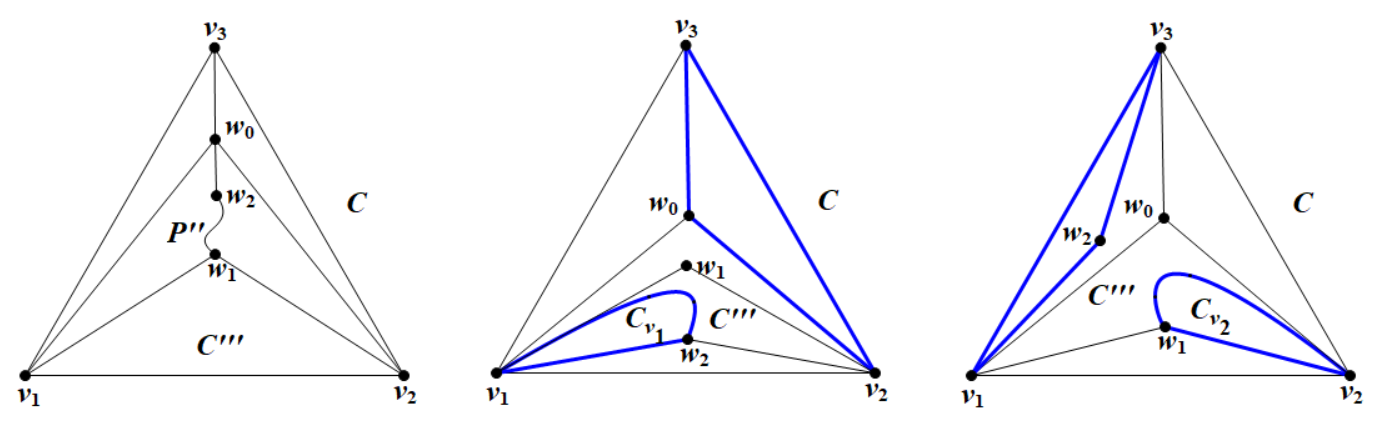}
	\caption{Three possible structures in Case 1. }{\label{fig.-04}}
\end{figure}

\noindent{\bf{Subcase 1.1.}} $|B_2|=1$, that is, $B_2=\{w_1\}$.

For each pendant vertex $w$ of $G_1$,
we have $d_C(w)=d_G(w)-d_{G_1}(w)\geq2$, consequently, $w\in B_2\cup B_3=\{w_1,w_0\}$.
This indicates that $G_1$ contains exactly two pendant vertices $w_1$ and $w_0$.
Furthermore, we can see that $G_1$ contains no isolated vertices
(otherwise, every isolated vertex of $G_1$ has at least three neighbors in $V(C)$ and so belongs to $B_3$,
while the unique vertex $w_0\in B_3$ is a pendant vertex of $G_1$).
Therefore, $G_1$ is a path of order $n-|C|$ with endpoints $w_1$ and $w_0$.

Now we know that $G_1$ is a path with $|G_1|=n-3\geq3$.
Let $N_{G_1}(w_0)=\{w_2\}$ and $P''=G_1-\{w_0\}$.
Then $P''$ is a path with endpoints $w_1$ and $w_2$.
Since $d_{G_1}(w_2)=2$, we have $d_{C}(w_2)\geq 1$.
If $w_2v_3\in E(G)$,
then $G$ contains two vertex-disjoint cycles $v_{3}w_0w_2v_{3}$ and $w_1v_{1}v_{2}w_1$, as desired.
If $w_2v_i\in E(G)$ for some $i\in \{1,2\}$,
then $G$ contains two vertex-disjoint cycles $v_iw_1P''w_2v_i$ and $w_0v_{\overline{i}}v_3w_0$, as desired.

\noindent{\bf{Subcase 1.2.}} $|B_2|\geq 2$.

Let $w_2\in B_2\setminus \{w_1\}$.
If $N_{C}(w_1)=N_{C}(w_2)$, then we may assume that $w_2\in int(C''')$
by the symmetry of $w_1$ and $w_2$, where $C'''=w_1v_1v_2w_1$.
By Claim \ref{claim4.4},
$G$ contains a cycle $C_{v_1}$ such that $V(C_{v_1})\subseteq Int(C''')$ and $V(C_{v_1})\cap V(C)=\{v_1\}$.
On the other hand,
Claim \ref{claim4.3} implies that $w_0\in ext(C''')$. Hence, $w_0\notin V(C_{v_1})$.
Therefore, $G$ contains two vertex-disjoint cycles $C_{v_1}$ and $w_0v_2v_3w_0$, as desired.

It remains the case that $N_{C}(w_1)\neq N_{C}(w_2)$.
Now $N_{C}(w_2)=\{v_i,v_3\}$ for some $i\in\{1,2\}$.
We define $C'''=w_0v_1v_2w_0$ instead of the original one in Claim \ref{claim4.4}.
Then $w_1\in int(C''')$. Moreover, $w_2\in ext(C''')$ as $w_2v_3\in E(G)$.
By Claim \ref{claim4.4}, there exists a cycle $C_{v_{\overline{i}}}$
such that $V(C_{v_{\overline{i}}})\subseteq Int(C''')$ and $V(C_{v_{\overline{i}}})\cap V(C)=\{v_{\overline{i}}\}$.
Therefore, $G$ contains two vertex-disjoint cycles $C_{v_{\overline{i}}}$ and $w_2v_iv_3w_2$, as desired.

\vspace{2mm}
\noindent{\bf{Case 2.}} $|B_3|=0$.

Recall that $A=\{w\in V(G_1)~|~ v_1,v_2\in N_C(w)\}$.
Since $|B_3|=0$, we can see that $N_C(w)=N_C(w_1)=\{v_1,v_2\}$ for each $w\in A$.
We may assume without loss of generality that $A\subseteq Int(C''')$
by the symmetry of vertices in $A$.
By Claim \ref{claim4.3}, there exists a pendant vertex $w_3$ of $G_1$ in $ext(C''')$,
which implies that $d_C(w_3)\geq2$.
Since $|B_3|=0$, we have $w_3\in B_2$, and thus $B_2\supseteq\{w_1,w_3\}$.
Moreover, $w_3\notin A$ as $A\subseteq Int(C''')$.
Assume without loss of generality that $N_{C}(w_3)=\{v_1,v_3\}$ (see Figure \ref{fig.-05}).
We also consider two subcases according to
$|B_2|$.

\begin{figure}[!ht]
	\centering
	\includegraphics[width=0.7\textwidth]{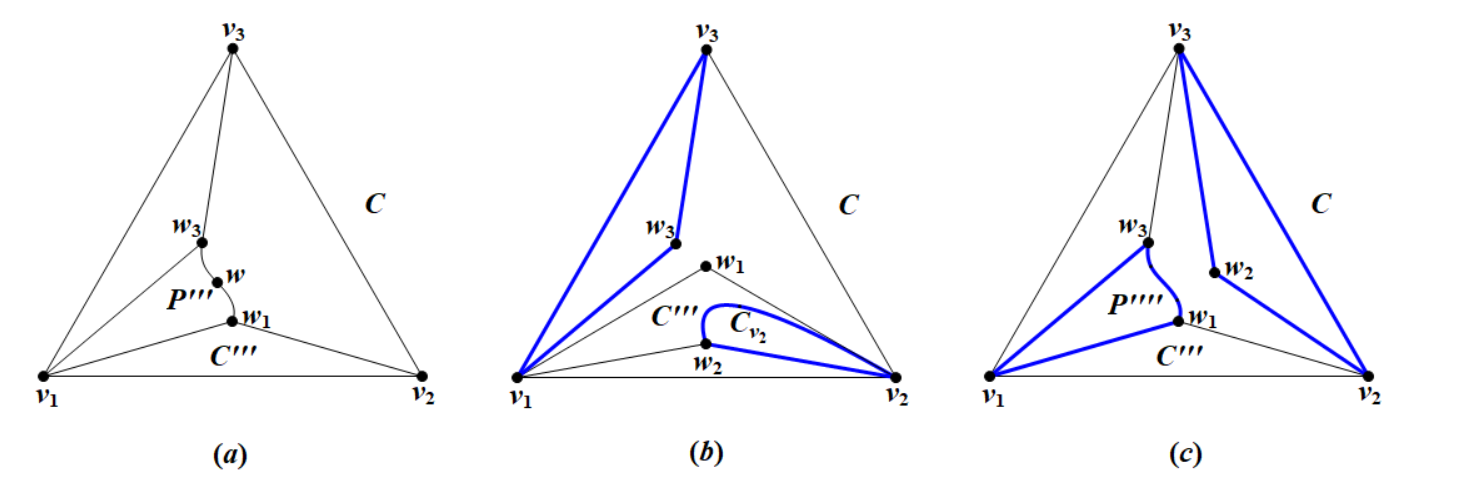}
	\caption{Three possible structures in Case 2.}{\label{fig.-05}}
\end{figure}

\noindent{\bf{Subcase 2.1.}}  $|B_2|=2$, that is, $B_2=\{w_1,w_3\}$.

Since $d_C(w_3)=2$, we have $d_{G_1}(w_3)\geq1$, which implies that
$G_1$ is non-empty and has at least two pendant vertices.
On the other hand, since $\delta(G)\geq3$ while $B_3=\varnothing$,
we can see that $G_1$ contains no isolated vertices,
and $w\in B_2=\{w_1,w_3\}$ for each pendant vertex $w$ of $G_1$.
Therefore, $G_1$ contains exactly two pendant vertices $w_1$ and $w_3$,
more precisely, $G_1$ is a path with endpoints $w_1$ and $w_3$.
Let $w$ be an arbitrary vertex in $V(G_1)\setminus \{w_1,w_3\}$.
Then, $d_C(w)=d_G(w)-d_{G_1}(w)=d_G(w)-2\geq 1.$

If $wv_2\in E(G)$, then $G$ contains two vertex-disjoint cycles $v_2w_1P'''wv_2$ and $v_1w_3v_3v_1$,
where $P'''$ is the subpath of $G_1$ from $w_1$ to $w$ (see Figure \ref{fig.-05}($a$)).
If $wv_3\in E(G)$, then $G$ contains two vertex-disjoint cycles $v_3w_3P'''wv_3$ and $v_1w_1v_2v_1$,
where $P'''$ is the subpath of $G_1$ from $w_3$ to $w$  (see Figure \ref{fig.-05}($a$)).
If $N_C(w)=\{v_1\}$ for each $w\in V(G_1)\setminus \{w_1,w_3\},$
then $G\cong K_1+C_{n-1}$, as desired.

\vspace{2mm}
\noindent{\bf{Subcase 2.2.}}   $|B_2|\geq 3$.

For each vertex $w\in B_2$,
it is clear that $N_C(w)$ is one of $\{v_1,v_2\}$, $\{v_1,v_3\}$ and $\{v_2,v_3\}$.
We first consider the case that there exist two vertices in $B_2$ which have the same neighbors in $C$.
Without loss of generality, assume that we can find a vertex $w_2\in B_2$ with $N_C(w_2)=N_C(w_1)=\{v_1,v_2\}$.
Then $w_2\in A$. Recall that $A\subseteq Int(C''')$ and $C'''=w_1v_1v_2w_1$ (see Figure \ref{fig.-05}($b$)).
Then, we can further get that $w_2\in int(C''')$.
By Claim \ref{claim4.4},
there exists a cycle $C_{v_2}$ such that $V(C_{v_2})\subset Int(C''')$ and
$V(C_{v_2})\cap V(C)=\{v_2\}$.
On the other hand,
Claim \ref{claim4.3} implies that $w_3\in ext(C''')$. Hence, $w_3\notin V(C_{v_2})$.
Therefore, $G$ contains two vertex-disjoint cycles $C_{v_2}$ and $w_3v_1v_3w_3$, as desired.

Now it remains the case that
any two vertices in $B_2$ have different neighborhoods in $C$.
This implies that $|B_2|=3$ and we can find a vertex $w_2\in B_2$ with $N_{C}(w_2)=\{v_2,v_3\}$.
Now we have $B_2=\{w_1,w_2,w_3\}$.
Furthermore,
since $\delta(G)\geq 3$ and $B_3=\varnothing$, we have $d_{G_1}(w)\geq1$ for each $w\in V(G_1)$,
and if $d_{G_1}(w)=1$, then $w\in B_2$.
Since $|B_2|=3$, we can see that $G_1$ has only one connected component,
that is, $G_1$ is a tree and some $w_i$, say $w_2$, is a pendant vertex of $G_1$.
Now, $G_1-\{w_2\}$ contains a subpath $P''''$ with endpoints $w_1$ and $w_3$  (see Figure \ref{fig.-05}($c$)).
Then $G$ contains two vertex-disjoint cycles $v_1w_1P''''w_3v_1$ and $w_2v_2v_3w_2$, as desired.

This completes the proof of Lemma \ref{lemma4.3}.
\end{proof}

Let $\mathcal{G}^*_n$ be the family of graphs obtained from $2K_1+C_3$ and an independent set of size $n-5$ by
joining each vertex of the independent set to arbitrary two vertices of the triangle
(see Figure \ref{fig.02}). Clearly, every graph in $\mathcal{G}^*_n$ is planar.
Now, let $\mathcal{G}_n$ be the family of planar graphs obtained from $2K_1+C_3$ by iteratively adding vertices of degree 2 until the resulting graph has $n$ vertices.
Then $\mathcal{G}^*_n\subseteq \mathcal{G}_n$.

\begin{figure}[!ht]
	\centering
	\includegraphics[width=0.3\textwidth]{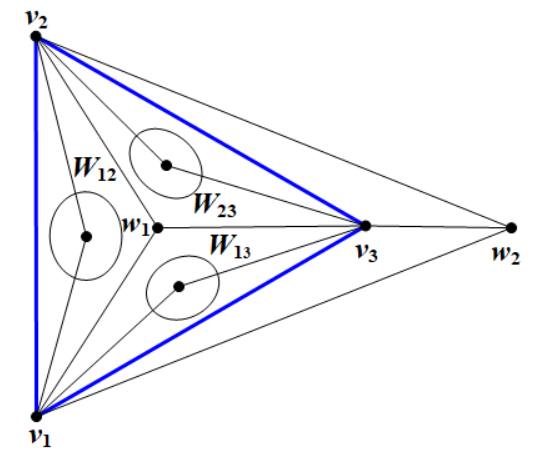}
	\caption{An extremal graph in $\mathcal{G}^*_n$.}{\label{fig.02}}
\end{figure}

\begin{lem}\label{lemma4.4}
For any graph $G\in\mathcal{G}_n$, $G$ is $2\mathcal{C}$-free if and only if $G\in\mathcal{G}^*_n$.
\end{lem}

\begin{proof}

Let  $V_1:=\{v_1,v_2,v_3\}$  be the set of vertices of degree 4 and
$V_2:=\{w_1,w_2\}$ be the set of vertices of degree 3 in $2K_1+C_3$, respectively.
Then $V_1$ induces a triangle.
We first show that every graph $G$ in $\mathcal{G}^*_n$ is $2\mathcal{C}$-free.
It suffices to prove that every cycle of $G$ contains at least two vertices in $V_1$.
Let $C$ be an arbitrary cycle of $G$.
If $V(C)\subseteq V_1$, then there is nothing to prove.
It remains the case that there exists a vertex $w\in V(C)\setminus V_1$.
By the definition of $\mathcal{G}^*_n$, we can see that $N_C(w)\subseteq N_{G}(w)\subseteq V_1$.
Note that $|N_C(w)|\geq2$. Hence, $C$ contains at least two vertices in $V_1$.

In the following, we will show that every graph $G\in \mathcal{G}_n\setminus \mathcal{G}^*_n$ contains two vertex-disjoint cycles.
By the definition of $\mathcal{G}_n$,
$G$ is obtained from $2K_1+C_3$ by iteratively adding $n-5$ vertices $u_1,u_2,\dots,u_{n-5}$ of degree 2.
Now, let $G_{n-5}=G$, and $G_{i-1}=G_i-\{u_i\}$ for $i\in\{1,2,\ldots,n-5\}$.
Then $G_0\cong 2K_1+C_3$.
Moreover, $|G_i|=i+5$ and $d_{G_i}(u_{i})=2$ for each $i\in \{1,2,\dots,n-5\}$.
Now let
$$i^*=\max\{i~|~0\leq i \leq n-5,~ G_i\in \mathcal{G}^*_{i+5}\}.$$
Since $G_0=2K_1+C_3\in \mathcal{G}^*_{5}$ and $G_{n-5}\notin \mathcal{G}^*_n$, we have $0\leq i^*\leq n-6$.
By the choice of $i^*$, we know that $G_{i^*}\in \mathcal{G}^*_{i^*+5}$ and $G_{i^*+1}\notin \mathcal{G}^*_{i^*+6}$,
which implies that $N_{G_{i^*}}(u_i)\subseteq V_1$ and $N_{G_{i^*+1}}(u_{i^*+1})\not\subseteq V_1$.

\begin{figure}[!ht]
	\centering
	\includegraphics[width=0.8\textwidth]{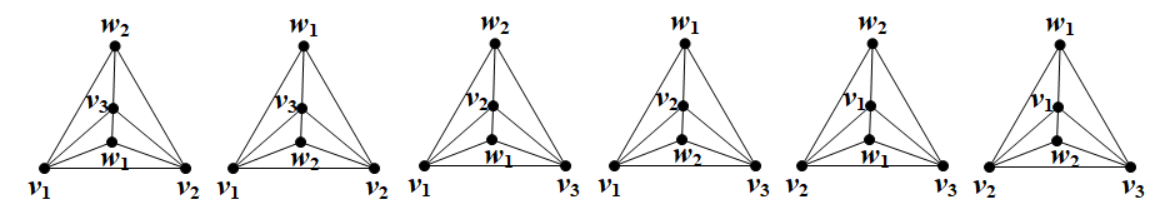}
	\caption{An extremal graph in $\mathcal{G}^*_n$.}{\label{fig.10}}
\end{figure}

Now we may assume that $G_{n-5}$ is a planar embedding of $G$,
and $G_0$ is a plane subgraph of $G_{n-5}$.
Observe that $2K_1+C_3$ has six planar embeddings (see Figure \ref{fig.10}).
Without loss of generality, assume that $G_0$ is the leftmost graph in Figure \ref{fig.10}.
Then, $u_{i^*+1}$ lies in one of the following regions (see Figure \ref{fig.10}):
$$ext(w_2v_1v_2w_2), int(w_2v_1v_3w_2), int(w_2v_2v_3w_2),$$
 $$int(w_1v_1v_2w_1), int(w_1v_1v_3w_1), int(w_1v_2v_3w_1).$$
By Lemma \ref{lemma4.2}, we can assume that $u_{i^*+1}$ lies in the outer face, that is,
$u_{i^*+1}\in ext(w_2v_1v_2w_2)$.
For simplify, we denote $C'=w_2v_1v_2w_2$.
Let $u$ be an arbitrary vertex with $uv_3\in E(G_{i^*+1})$.
Then by Lemma \ref{lemma4.1} and $v_3\in int(C')$, we have $u\in int(C')$,
and thus $uu_{i^*+1}\notin E(G_{i^*+1})$.
This implies that $N_{G_{i^*+1}}(u_{i^*+1})\subseteq V(C')\cup W_{12}$,
where $W_{12}=\{u~|~u\in V(G_{i^*+1}),~N_{G_{i^*+1}}(u)=\{v_1,v_2\}\}$.
Recall that $d_{G_{i^*+1}}(u_{i^*+1})=2$ and $N_{G_{i^*+1}}(u_{i^*+1})\not\subseteq V_1$.
Then, $|N_{G_{i^*+1}}(u_{i^*+1})\cap\{v_1,v_2\}|\le 1$.
If $|N_{G_{i^*+1}}(u_{i^*+1})\cap\{v_1,v_2\}|= 1$,
then we may assume without loss of generality that $v_1\in N_{G_{i^*+1}}(u_{i^*+1})$, and
$u'\in N_{G_{i^*+1}}(u_{i^*+1})\setminus \{v_1\}$.
Since $N_{G_{i^*+1}}(u_{i^*+1})\subseteq V(C')\cup W_{12}$,
we have $u'\in \{w_2\}\cup W_{12}$, and so $u'v_1\in E(G_{i^*+1})$.
Thus, $G_{n-5}$ contains two vertex-disjoint cycles $u_{i^*+1}v_1u'u_{i^*+1}$ and $w_1v_2v_3w_1$,
as desired.
Now consider the case that $| N_{G_{i^*+1}}(u_{i^*+1})\cap\{v_1,v_2\}|=0$.
This implies that $N_{G_{i^*+1}}(u_{i^*+1})\subseteq \{w_2\}\cup W_{12}$.
Let $N_{G_{i^*+1}}(u_{i^*+1})=\{u',u''\}$.
Then $u'v_1,u''v_1\in E(G_{i^*+1})$. Therefore,
$G_{n-5}$ contains two vertex-disjoint cycles $u_{i^*+1}u'v_1u''u_{i^*+1}$ and $w_1v_2v_3w_1$.
\end{proof}

Given a graph $G$,
let $\widetilde{G}$ be the largest induced subgraph of $G$ with minimal degree at least 3.
It is easy to see that
$\widetilde{G}$ can be obtained from $G$ by iteratively removing the vertices of degree at most 2 until the resulting graph has minimum degree at least 3 or is empty.
It is well known that $\widetilde{G}$ is unique and does not depend on the order of vertex deletion (see \cite{Pittel}).

In the following, we give the proof of Theorem \ref{theorem1.4}.

\begin{proof}
Let $n\geq5$ and $G$ be an extremal graph corresponding to ${\rm ex}_{\mathcal{P}}(n,2\mathcal{C})$.
Observe that $K_2+(P_3\cup (n-5)K_1)$ is a planar graph which contains no two vertex-disjoint cycles (see Figure \ref{fig.02}).
Thus, $e(G)\geq e(K_2+(P_3\cup (n-5)K_1))=2n-1$.

If $\widetilde{G}$ is empty, then we define $G'$ as
the graph obtained from $G$ by iteratively removing the vertices of degree at most 2
until the resulting graph has 4 vertices.
By the planarity of $G'$, we have $e(G')\leq 3|G'|-6=6$, and thus
$$e(G)\le e(G')+2(n-4)\leq6+2n-8=2n-2,$$
a contradiction.

Now we know that $\widetilde{G}$ is nonempty.
Then, $\widetilde{G}$ contains no two vertex-disjoint cycles as $\widetilde{G}\subseteq G$.
By the definition of $\widetilde{G}$, we have $\delta(\widetilde{G})\geq 3$.
By Lemma \ref{lemma4.3}, we get that $\widetilde{G}\in \{2K_1+C_3,K_1+C_{|\widetilde{G}|-1}\}$.
If $\widetilde{G}\cong K_1+C_{|\widetilde{G}|-1}$, then
$$e(G)\le e(\widetilde{G})+2(n-|\widetilde{G}|)=2(|\widetilde{G}|-1)+2(n-|\widetilde{G}|)=2n-2,$$
a contradiction. Thus, $\widetilde{G}\cong 2K_1+C_3$.
Now, $e(G)\le e(\widetilde{G})+2(n-5)=2n-1$.
Therefore, $e(G)=2n-1$,
which implies that ${\rm ex}_{\mathcal{P}}(n,2\mathcal{C})=2n-1$ and $G\in \mathcal{G}_n$.
By Lemma \ref{lemma4.4}, we have $G\in \mathcal{G}^*_n$.
This completes the proof of Theorem \ref{theorem1.4}.
\end{proof}
\section{Proof of Theorem \ref{theorem1.5}}
We shall further introduce some notations on a plane graph $G$.
A vertex or an edge of $G$ is said to be \emph{incident} with a face $F$, if
it lies on the boundary of $F$.
Clearly, every edge of $G$ is incident with at most two faces.
A face of size $i$ is called an $i$-face.
The numbers of $i$-faces and total faces are denoted by $f_i(G)$ and $f(G)$, respectively.
Let $E_3(G)$ be the set of edges incident with at least one $3$-face,
and particularly, let $E_{3,3}(G)$ be the set of edges incident with two $3$-faces.
Moreover, let $e_{3}(G)$ and $e_{3,3}(G)$ denote the cardinalities of
$E_{3}(G)$ and $E_{3,3}(G)$, respectively.
We can easily see that $3f_3(G)=e_{3}(G)+e_{3,3}(G).$

Lan, Shi and Song proved that ${\rm ex}_{\mathcal{P}}(n,K_1+P_3)\le \frac{12(n-2)}{5}$,
with equality when $n\equiv12 \pmod{20}$ (see \cite{LS}), and
${\rm ex}_{\mathcal{P}}(n,K_1+P_{k+1})\le \frac{13kn}{4k+2}-\frac{12k}{2k+1}$ for $ k\in \{3,4,5\}$ (see \cite{D}).
For $k\geq 6$,  one can easily see that ${\rm ex}_{\mathcal{P}}(n,K_1+P_{k+1})=3n-6$.
In \cite{Fang}, the authors obtained the following sharp result.

\begin{lem}\label{lemma5.1}\emph{(\cite{Fang})}
Let $n,k$ be two integers with $k\in \{2,3,4,5\}$ and $n\geq \frac{12}{6-k}+1$.
Then ${\rm ex}_{\mathcal{P}}(n,K_1+P_{k+1})\le \frac{24k}{7k+6}(n-2)$,
with equality if $n\equiv{\frac{12(k+2)}{6-k}}\pmod{{\frac{28k+24}{6-k}}}$.
\end{lem}

To prove Theorem \ref{theorem1.5}, we also need an edge-extremal result on outerplanar graphs.
Let ${\rm ex}_{\mathcal{OP}}(n,C_k)$ denote the maximum number of edges
in an $n$-vertex $C_k$-free outerplanar graph.

\begin{lem}\label{lemma5.2} \emph{(\cite{Fang2})}
 Let $n,k,\lambda$ be three integers with $n\geq k\geq 3$ and $\lambda=\lfloor\frac{kn-2k-1}{k^2-2k-1}\rfloor+1$.
 Then
 $${\rm ex}_{\mathcal{OP}}(n,C_k)=\left\{
                                       \begin{array}{ll}
                                         2n-\lambda+2\big\lfloor\frac{\lambda}{k}\big\rfloor-3  & \hbox{if $k\mid \lambda$,} \\
                                         2n-\lambda+2\big\lfloor\frac{\lambda}{k}\big\rfloor-2  & \hbox{otherwise.}
                                       \end{array}
                                     \right.
$$
\end{lem}

In particular, we can obtain the following corollary.

\begin{cor}\label{cor5.1}
Let $n\geq 4$. Then
  $${\rm ex}_{\mathcal{OP}}(n-1,C_4)=\left\{
                          \begin{array}{ll}
                          \frac{12}{7}n-5 & \hbox{if $7\mid n$,} \\
                         \big\lfloor\frac{12n-27}{7}\big\rfloor~~~ & \hbox{otherwise.}
                          \end{array}
                          \right.
$$
\end{cor}

\begin{figure}[!ht]
	\centering
	\includegraphics[width=0.9\textwidth]{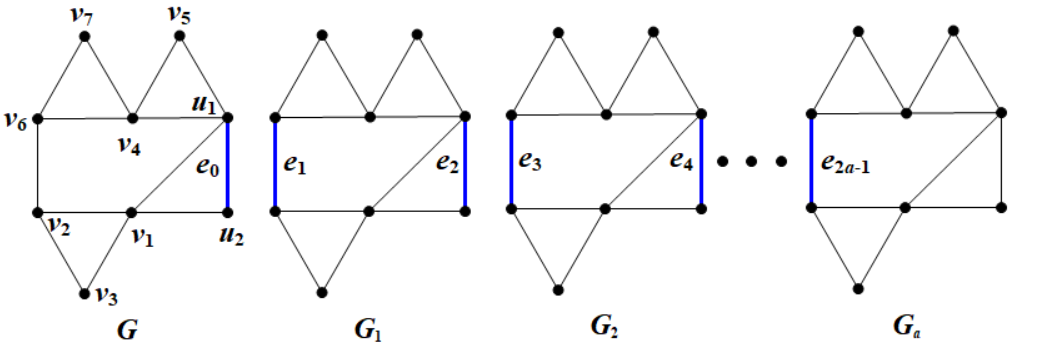}
	\caption{The constructions of $G,G_1,G_2,\ldots,G_a$. }{\label{figu.2}}
\end{figure}

For arbitrary integer $n\geq 4$, we can find a unique $(a,b)$ such that $a\geq 0$, $1\le b \le 7$ and $n-1=7a+b+2$.
Let $G$ be a $9$-vertex outerplanar graph and $G_1,\ldots,G_a$ be $a$ copies of $G$ (see Figure \ref{figu.2}).
Then, we define $G_0$ as the subgraph of $G$ induced by $\{u_1,u_2\}\cup \{v_1,v_2,\dots,v_b\}$.
One can check that $|G_0|=b+2$ and
$$ e(G_0)= \left\{
                          \begin{array}{ll}
                          \frac{12(b+2)-23}{7} & \hbox{if $7\mid (b+2-6)$,} \\
                         \big\lfloor\frac{12(b+2)-15}{7}\big\rfloor~~~ & \hbox{otherwise.}
                          \end{array}
                          \right.$$
We now construct a new graph $G^*$ from $G_0,G_1,\ldots,G_a$
by identifying the edges $e_{2i}$ and $e_{2i+1}$ for each $i\in \{0,\ldots,a-1\}$.
Clearly, $G^*$ is a connected $C_4$-free outerplanar graph with $$|G^*|=\sum_{i=0}^{a}|G_i|-2a=(2+b)+9a-2a=n-1.$$
Moreover, since $n\equiv b+2-6\,(\!\!\!\mod 7),$ we have
$$ e(G^*)= \sum_{i=0}^{a}e(G_i)-a=e(G_0)+12a=\left\{
                          \begin{array}{ll}
                          \frac{12}{7}n-5 & \hbox{if $7\mid n$,} \\
                         \big\lfloor\frac{12n-27}{7}\big\rfloor~~~ & \hbox{otherwise.}
                          \end{array}
                          \right.$$
Combining Corollary \ref{cor5.1},
$G^*$ is an extremal graph corresponding to ${\rm ex}_{\mathcal{OP}}(n-1,C_4)$.

\begin{lem}\label{lemma5.3}
Let $n\geq 2661$ and $G^{**}$ be an extremal plane graph corresponding to ${\rm ex}_{\mathcal{P}}(n,2C_4)$.
Then $G^{**}$ contains at least fourteen quadrilaterals and all of them share exactly one vertex.
\end{lem}

\begin{proof}
Note that $e(G^*)={\rm ex}_{\mathcal{OP}}(n-1,C_4)\geq \frac{12}{7}n-5$
and $G^*$ is a $C_4$-free outerplanar graph of order $n-1$.
Then $K_1+G^*$ is an $n$-vertex $2C_4$-free planar graph, and thus
\begin{align*}
 e(G^{**})\geq e(K_1+G^*)=e(G^*)+n-1\geq\frac{19}{7}n-6.
\end{align*}
On the other hand, by Lemma \ref{lemma5.1}, we have
$${\rm ex}_{\mathcal{P}}(n,K_1+P_4)\le \frac{8}{3}(n-2).$$
Note that $\frac{19}{7}n-6>\frac{8}{3}(n-2)$ for $n\geq 2661$.
Then $G^{**}$ contains a copy, say $H_1$, of $K_1+P_4$.
Let $G_1$ be the graph obtained from $G^{**}$ by deleting all edges within $V(H_1)$.
Since $|H_1|=5$, we have $e(G_1)\geq e(G^{**})-(3|H_1|-6)=e(G^{**})-9>\frac{8}{3}(n-2)$.
Thus, $G_1$ contains a copy, say $H_2$, of $K_1+P_4$. Now we can obtain a new graph $G_2$
from $G_1$ by deleting all edges within $V(H_2)$.
Note that $e(G^{**})-14\times9>\frac{8}{3}(n-2)$.
Repeating above steps, we can obtain a graph sequence $G_1,G_2,\ldots,G_{14}$
and fourteen copies $H_1,H_2,\cdots,H_{14}$ of $K_1+P_4$ such that $H_i\subseteq G_{i-1}$
and $G_{i}$ is obtained from $G_{i-1}$ by deleting all edges within $V(H_i)$.
This also implies that $G^{**}$ contains at least fourteen quadrilaterals.
We next give four claims on those copies of $K_1+P_4$.

\begin{claim}\label{claim5.1}
Let $i,j$ be two integers with $1\le i<j\le 14$ and $v\in V(H_i)\cap V(H_j)$.
Then, $V(H_i)\cap N_{H_j}(v)=\varnothing$.
\end{claim}

\begin{proof}
Suppose to the contrary that there exists a vertex $w\in V(H_i)\cap N_{H_j}(v)$.
Note that $v,w\in V(H_i)$.
By the definition of $G_i$, whether $vw\in E(H_i)$ or not, we can see that $vw\notin E(G_i)$.
On the other hand,
note that $H_j\subseteq G_{j-1}\subseteq G_i$, then
$vw\in E(H_j)\subseteq E(G_i)$, contradicting $vw\notin E(G_i).$
Hence, the claim holds.
\end{proof}

\begin{claim}\label{claim5.2}
$|V(H_i)\cap V(H_j)|\in\{1,2\}$ for any two integers $i,j$ with $1\le i<j\le 14$.
\end{claim}

\begin{proof}
If $H_i$ and $H_j$ are vertex-disjoint, then $G^{**}$ contains $2C_4$, a contradiction.
Now suppose that there exist three vertices $v_1,v_2,v_3\in V(H_i)\cap V(H_j)$.
Observe that $K_1+P_4$ contains no an independent set of size $3$.
Then $H_j[\{v_1,v_2,v_3\}]$ is nonempty.
Assume without loss of generality that $v_1v_2\in E(H_j)$.
Then $v_2\in V(H_i)\cap N_{H_j}(v_1)$, which contradicts Claim \ref{claim5.1}.
Therefore, $1\le |V(H_i)\cap V(H_j)|\le 2$.
\end{proof}

Now for convenience, a vertex $v$ in a graph $G$ is called a \emph{$2$-vertex}
if $d_G(v)=2$, and a \emph{$2^+$-vertex}
if $d_G(v)>2.$ Clearly, every copy of $K_1+P_4$ contains two 2-vertices and three $2^+$-vertices.

\begin{claim}\label{claim5.3}
Let $\mathcal{H}$ be the family of graphs $H_i$ ($1\leq i\leq14$) such that
every 2-vertex in $H_i$ is a $2^+$-vertex in $H_1$.
Then $|\mathcal{H}|\le 3$.
\end{claim}

\begin{proof}
Note that $H_1$ contains only three $2^+$-vertices, say $v_1,v_2$ and $v_3$.
Then every graph $H_i\in \mathcal{H}$ must contain two of $v_1,v_2$ and $v_3$ as $2$-vertices.
Suppose to the contrary that $|\mathcal{H}|\geq 4$.
By pigeonhole principle, there exist two graphs $H_{i_1},H_{i_2}\in \mathcal{H}$ such that
they contain the same two 2-vertices, say $v_1,v_2$.
It follows that $H_{i_j}-\{v_j\}$ contains a $4$-cycle for $j\in \{1,2\}$.
By Claim \ref{claim5.2}, we have $V(H_{i_1})\cap V(H_{i_2})=\{v_1,v_2\}$,
which implies that $H_{i_1}-\{v_1\}$ and $H_{i_2}-\{v_2\}$ are vertex-disjoint.
Hence, $G^{**}$ contains two vertex-disjoint 4-cycles, a contradiction.
\end{proof}

\begin{claim}\label{claim5.4}
Let $j$ be an integer with $2\leq j\leq14$ and $H_j\notin \mathcal{H}$.
Then, there exists a vertex $v\in V(H_1)\cap V(H_{j})$ such that $d_{H_1}(v)\geq 3$ and $d_{H_{j}}(v)\geq 3$.
\end{claim}

\begin{proof}
By Claim \ref{claim5.2}, we have $1\le |V(H_1)\cap V(H_j)|\le 2$.
We first assume that $V(H_1)\cap V(H_j)=\{u\}$.
If $d_{H_1}(u)\geq 3$ and $d_{H_j}(u)\geq3$, then there is nothing to prove.
If $d_{H_1}(u)=2$, then
$G^{**}$ contains two vertex-disjoint subgraphs $H_1-\{u\}$ and $H_j$, and thus $2C_4$,
a contradiction.
If $d_{H_j}(u)=2$, then we can similarly get a contradiction.
Therefore, $|V(H_1)\cap V(H_j)|=2$.

Now, assume that $V(H_1)\cap V(H_j)=\{u_1,u_2\}$.
We first deal with the case $d_{H_j}(u_1)$ $=d_{H_j}(u_2)=2$.
Since $H_j\notin \mathcal{H}$, one of $\{u_1,u_2\}$, say $u_1$, is a 2-vertex in $H_1$.
Hence, $G^{**}$ contains two vertex-disjoint subgraphs $H_1-\{u_1\}$ and $H_j-\{u_2\}$, and so $2C_4$,
a contradiction.
Thus, there exists some $i\in \{1,2\}$ with $d_{H_j}(u_i)\geq 3$.
If $d_{H_1}(u_i)\geq 3$, then we are done.
If $d_{H_1}(u_i)=2$, then
we define $H_j'$ as the subgraph of $H_j$ induced by $N_{H_j}(u_i)\cup\{u_i\}$.
Since $d_{H_j}(u_i)\geq 3$, we can check that $H_j'$ always contains a $C_4$.
Moreover, since $d_{H_1}(u_i)=2$, we can see that $H_1-\{u_i\}$ also contains a $C_4$.
On the other hand, by Claim \ref{claim5.1},
we have $N_{H_j}(u_i)\cap V(H_1)=\varnothing$,
which implies that $H_j'$ and $H_1-\{u_i\}$ are vertex-disjoint.
Therefore, $G^{**}$ contains $2C_4$, a contradiction.
\end{proof}

By Claim \ref{claim5.3}, $|\mathcal{H}|\leq 3$, thus there are at least ten graphs
in $\{H_j~|~2\leq j \leq 14\}\setminus \mathcal{H}$.
However, $H_1$ has only three $2^+$-vertices.
By Claim \ref{claim5.4} and pigeonhole principle, there exists a $2^+$-vertex $w$ in $H_1$
and four graphs, say $H_2,H_3,H_4,H_5$, of $\{H_j~|~2\leq j \leq 14\}\setminus \mathcal{H}$.
By Claim \ref{claim5.1},
we get that $N_{H_j}(w)\cap V(H_i)=\varnothing$, and so $N_{H_j}(w)\cap N_{H_i}(w)=\varnothing$, for any $i,j$ with $1\leq i<j\leq 5$.
If $G^{**}-\{w\}$ contains a quadrilateral $C'$,
then there exists some $j'\le 5$ such that $N_{H_{j'}}(w)\cap V(C')=\varnothing$ as $|C'|=4$.
Since $w$ is a $2^+$-vertex in $H_{j'}$, we can observe that the subgraph of $H_{j'}$ induced by $N_{H_{j'}}(w)\cup\{w\}$ must contain a $C_4$.
Consequently, $G^{**}$ is not $2C_4$-free, a contradiction.
Thus, $G^{**}-\{w\}$ is $C_4$-free,
which implies that all quadrilaterals of $G^{**}$ share exactly one vertex.
This completes the proof of Lemma \ref{lemma5.3}.
\end{proof}

Now we are ready to give the proof of Theorem \ref{theorem1.5}.

\begin{proof}
Recall that $G^*$ is an extremal graph corresponding to ${\rm ex}_{\mathcal{OP}}(n-1,C_4)$.
Then $K_1+G^*$ is planar and $2C_4$-free.
By Corollary \ref{cor5.1}, we have
\begin{align}\label{align.-1}
 e(K_1+G^*)={\rm ex}_{\mathcal{OP}}(n-1,C_4)+n-1=\left\{
                          \begin{array}{ll}
                          \frac{19}{7}n-6 & \hbox{if $7\mid n$,} \\
                         \big\lfloor\frac{19n-34}{7}\big\rfloor~ & \hbox{otherwise.}
                          \end{array}
                          \right.
\end{align}
To prove Theorem \ref{theorem1.5},
it suffices to show ${\rm ex}_{\mathcal{P}}(n,2C_4)=e(K_1+G^*).$
Since $G^{**}$ is an extremal plane graph corresponding to ${\rm ex}_{\mathcal{P}}(n,2C_4)$,
we have $e(G^{**})\geq e(K_1+G^*)$.
In the following, we show that $e(G^{**})\le e(K_1+G^*)$.

\begin{figure}[!ht]
	\centering
	\includegraphics[width=0.4\textwidth]{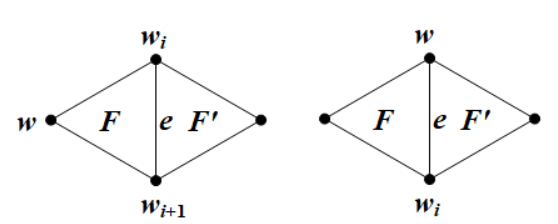}
	\caption{Two possible structures of $H(e)$. }{\label{fig.-11}}
\end{figure}

By Lemma \ref{lemma5.3},
all quadrilaterals of $G^{**}$ share a vertex $w$.
Thus, $G^{**}-\{w\}$ is $C_4$-free.
Assume that $d_{G^{**}}(w)=s$ and $w_1,\dots,w_s$ are
around $w$ in clockwise order, with subscripts interpreted modulo $s$.
Let $e$ be an arbitrary edge in $E_{3,3}(G^{**})$, that is,
$e$ is incident with two 3-faces, say $F$ and $F'$.
We define $H(e)$ as the plane subgraph induced by all edges incident with $F$ and $F'$.
Clearly, $H(e)\cong K_1+P_3$ and so it contains a $C_4$.
Recall that all quadrilaterals of $G^{**}$ share exactly one vertex $w$.
Then, $w\in V(H_e)$ and $w$ is incident with at least one face of $F$ and $F'$
(see Figure \ref{fig.-11}).
Note that $e$ is incident with $F$. Then,
either $e=ww_i$ or $e=w_iw_{i+1}$ for some $i\in \{1,2,\dots,s\}$.
By the choice of $e$, we have
\begin{align}\label{align.-02}
 E_{3,3}(G^{**})\subseteq \{ww_{i},w_iw_{i+1}~|~1\leq i\leq s\}.
\end{align}

Assume first that $f_4(G^{**})=t\geq1$ and $F_1,\dots,F_t$ are 4-faces in $G^{**}$.
Since every 4-face is a quadrilateral,
$w$ is incident with each 4-face.
Consequently,
there exists $j_i\in \{1,\dots,s\}$ such that $ww_{j_i},ww_{j_i+1}$ are incident with $F_i$ for each $i\in \{1,\dots,t\}$.
Thus, $ww_{j_i}\notin E_{3,3}(G^{**})$ for $1\le i \le t$.
On the other hand,
if $w_{j_i}w_{j_i+1}\in E_{3,3}(G^{**})$, then $H(w_{j_i}w_{j_i+1})$ contains a $C_4$,
and so $w\in V(H(w_{j_i}w_{j_i+1}))$.
This implies that $ww_{j_i}w_{j_i+1}w$ is a 3-face in $G^{**}$,
contradicting the fact that $ww_{j_i},ww_{j_i+1}$ are incident with the 4-face $F_i$.
Thus, we also have $w_{j_i}w_{j_i+1}\notin E_{3,3}(G^{**})$ for $1\le i \le t$.

By the argument above, we can see that
\begin{align}\label{align.-03}
 E_{3,3}(G^{**})\cap \{ww_{j_i},w_{j_i}w_{j_i+1}~|~1\le i \le t\}=\varnothing.
\end{align}
Using (\ref{align.-02}) and (\ref{align.-03}) gives $e_{3,3}(G^{**})\le 2s-2t=2s-2f_4(G^{**})$.
Hence,
\begin{align}\label{align.-04}
  3f_3(G^{**})=e_3(G^{**})+e_{3,3}(G^{**})\le e(G^{**})+2s-2f_4(G^{**}).
\end{align}
On the other hand,
\begin{align*}
2e(G^{**})=\sum_{i\geq 3}if_i(G^{**})\geq 3f_3(G^{**})+4f_4(G^{**})+5(f(G^{**})-f_3(G^{**})-f_4(G^{**})),
\end{align*}
which yields $f(G^{**})\le\frac{1}{5}\left(2e(G^{**})+2f_3(G^{**})+f_4(G^{**})\right).$
Combining this with Euler's formula $f(G^{**})=e(G^{**})-(n-2)$, we obtain
\begin{align}\label{align.3}
  e(G^{**})\le \frac{5}{3}(n-2)+\frac{2}{3}f_3(G^{**})+\frac{1}{3}f_4(G^{**}).
\end{align}
If $f_4(G^{**})=t=0$, then (\ref{align.-04}) and (\ref{align.3}) hold directly.
Combining (\ref{align.-04}) and (\ref{align.3}), we have $e(G^{**})\le \frac{15}{7}(n-2)+\frac{4}{7}s-\frac{1}{7}f_4(G^{**})$.
Recall that $d_{G^{**}}(w)=s\leq n-1$.
If $s\le n-2$, then $e(G^{**})\leq \lfloor\frac{19}{7}(n-2)\rfloor\leq e(K_1+G^*)$ by (\ref{align.-1}),
as desired.
If $s=n-1$,
then $w$ is a dominating vertex of the planar graph $G^{**}$,
which implies that $G^{**}-\{w\}$ is outerplanar.
Recall that $G^{**}-\{w\}$ is $C_4$-free.
Thus, $e(G^{**}-\{w\})\le {\rm ex}_{\mathcal{OP}}(n-1,C_4)$,
and so $e(G^{**})\leq {\rm ex}_{\mathcal{OP}}(n-1,C_4)+n-1$.
Combining (\ref{align.-1}), we get $e(G^{**})\leq e(K_1+G^*)$,
as required.
This completes the proof of Theorem \ref{theorem1.5}.
\end{proof}

\end{document}